\documentclass[reqno,a4paper,12pt]{article}

\usepackage[english]{babel}
\usepackage{amsmath,amssymb,amsfonts,amsthm,enumerate,mathrsfs}

\usepackage{epsf, epsfig}
\usepackage{graphicx}

\newtheorem{theorem}{Theorem}[section]
\newtheorem{lemma}[theorem]{Lemma}
\newtheorem{ajout}[theorem]{Complement}
\newtheorem{definition}[theorem]{Definition}
\newtheorem{remark}[theorem]{Remark}
\newtheorem{proposition}[theorem]{Proposition}
\newtheorem{corollary}[theorem]{Corollary}
\newtheorem{example}[theorem]{Example}

\newcommand{\Q}{\mathbb{Q}}
\newcommand{\N}{\mathbb{N}}
\newcommand{\R}{\mathbb{R}}

\newcommand{\Leg}{\mathscr{L}}
\newcommand{\KK}{\mathscr{K}}

\newcommand{\B}{\mathbb{B}}

\newcommand{\<}{\langle}
\renewcommand{\>}{\rangle}
\renewcommand{\a}{\alpha}

\renewcommand{\d}{\delta}

\newcommand{\e}{\varepsilon}
\newcommand{\g}{\gamma}

\newcommand{\s}{\sigma}

\renewcommand{\o}{\omega}

\renewcommand{\i}{\infty}
\newcommand{\p}{\partial}

\newcommand{\Id}{\operatorname{Id}}

\newcommand{\grad}{\operatorname{grad}}
\newcommand{\graph}{\operatorname{Graph}}
\title{Optimal transportation on non-compact manifolds}
\date{07 November 2007}
\author{ Albert Fathi
         \thanks{UMPA, ENS Lyon,
         46 All\'ee d'Italie,
         69007 Lyon, France.
     \textsf{e-mail: albert.fathi@umpa.ens-lyon.fr}},\ \ \ \ \
         Alessio Figalli
         \thanks{Scuola Normale Superiore,
     Piazza  dei Cavalieri 7, 56100 Pisa, Italy.
         \textsf{e-mail: a.figalli@sns.it}}}

\begin{document}

\maketitle

\begin{abstract} In this work, we show how to obtain for non-compact manifolds the results that have already been done for Monge Transport Problem for costs coming from Tonelli Lagrangians on compact manifolds. In particular, the already known results for a cost of the type
$d^r,r>1$, where $d$ is the Riemannian distance of a complete Riemannian manifold, hold without any curvature restriction.
\end{abstract}

\section{Introduction}
Monge transportation problem is more than 200 hundred years old, see \cite{monge}. It has generated a huge amount of work. It is impossible for us to give all the relevant references to previous works. There are now several books and surveys that can help the reader through the literature, see for example among others the books \cite{CaffarelliSalsa, amgisa, racrus,villani1,villani2}, and the surveys \cite{amb1, evans, gangbo}.
Both authors benefited a lot from all these sources.

Originally Monge
wanted to move, in 3-space,   rubble (d\'eblais) to build up a mound or fortification (remblais) minimizing the cost. Now if the rubble consist of masses $m_1,\dots, m_n$ at locations
$\{x_1,\dots x_n\}$, one should move them into another set of positions $\{y_1,\dots,y_n\}$ by minimizing the traveled distance taking into accounts the weights. Therefore one should try to minimize
\begin{equation}
\label{transpdist1}
\sum_{i=1}^n m_i d(x_i,T(x_i)),
\end{equation}
over all bijections $T:\{x_1,\dots x_n\}\to \{y_1,\dots,y_n\}$, where $d$ is the usual Euclidean distance on 3-space.

Nowadays, one would be more interested in minimizing the energy cost rather than the traveled distance. Therefore one would try rather to minimize
\begin{equation}
\label{transpdist2}
\sum_{i=1}^n m_i d^2(x_i,T(x_i)).
\end{equation}

Of course, one would like to generalize to continuous rather than just discrete distributions of matter. Therefore Monge transportation problem is now stated in the following general form: given two probability measures $\mu$ and $\nu$,
defined on the measurable spaces $X$ and $Y$, find a measurable map $T:X \rightarrow Y$ with

\begin{equation}
\label{transportcond}
T_\sharp \mu=\nu,
\end{equation}
i.e.
$$
\nu(A)=\mu\left(T^{-1}(A)\right) \ \ \ \forall A \subset Y \mbox{ measurable,}
$$
and in such a way that $T$ minimize the transportation cost. This last condition means
$$
\int_X c(x,T(x))\,d\mu(x) = \min_{S_\sharp \mu=\nu} \left\{ \int_X c(x,S(x))\,d\mu(x) \right\},
$$
where $c:X \times Y \rightarrow \R$ is some given cost function, and the minimum is taken over all measurable maps $S:X\to Y$ with $S_\sharp \mu=\nu$.
When condition (\ref{transportcond}) is satisfied, we say that $T$ is a \textit{transport map}, and if $T$ minimize also the cost we
call it an \textit{optimal transport map}.

It is easy to build examples where the Monge problem is ill-posed simply because there is no transport map:
this happens for instance when $\mu$ is a Dirac mass while $\nu$ is not. This means that one needs some restrictions on the measures $\mu$ and $\nu$.

Even in Euclidean spaces, and the cost $c$ equal to the Euclidean
distance or its square, the problem of the existence of an optimal
transport map is far from being trivial. Due to the strict
convexity of the square of the Euclidean distance, case
(\ref{transpdist2}) above is simpler to deal with than case
(\ref{transpdist1}). The reader should consult the books and
surveys given above to have a better view of the history of the
subject, in particular Villani's second book on the subject
\cite{villani2}. However for the case  where the cost is a
distance, like in  (\ref{transpdist1}), one should cite at least
the work of Sudakov \cite{sudakov}, Evans-Gangbo
\cite{evans-gangbo}, Feldman-McCann \cite{feldman-mccann},
Caffarelli-Feldman-McCann \cite{caffarelli-feldman-mccann},
Ambrosio-Pratelli \cite{ambrosio-pratelli}, and Bernard-Buffoni
\cite{bernbuf2}. For the case  where the cost is the square of the
Euclidean or of a  Riemannian distance, like in
(\ref{transpdist2}), one  should cite at least the work of
Knott-Smith \cite{knott-smith}, Brenier \cite{brenier},
Rachev-R\"uschendorf \cite{rachev-ruschendorf-article},
Gangbo-McCann \cite{gangbo-mccann}, McCann \cite{mccann3}, and
Bernard-Buffoni \cite{bernbuf}.

Our work is related to the case where the cost behaves like a square of a Riemannian distance. It is strongly inspired by the work of Bernard-Buffoni \cite{bernbuf}. In fact, we prove the non-compact version of this last work adapting some techniques that were first used in the euclidean case in \cite{amgisa} by Ambrosio,Gigli, and Savar\'e.
We show that the Monge transport problem can be solved for the square distance on any complete Riemannian manifold {\it without any assumption on the compactness or curvature}, with the usual restriction on the measures. Most of the arguments in this work are well known to specialists, at least in the compact case, but they have not been put together before and adapted to the case we treat. Of course, there is a strong intersection with some of the results that appear in \cite{villani2}.
For the case where the cost behaves like  the distance of a complete non-compact Riemannian manifold see the work \cite{figalli} of the second author.

We will prove a generalization of the following theorem (see Theorems \ref{thLagragiancost} and \ref{coutRiemannien2}):
\begin{theorem}\label{coutRiemannien}{\sl Suppose that $M$ is a connected complete Riemannian manifold, whose Riemannian distance is denoted by $d$. Suppose that $r>1$. If $\mu$ and $\nu$ are probability (Borel) measures on $M$, with $\mu$ absolutely continuous with respect to Lebesgue measure, and
$$
\int_Md^r(x,x_0)\,d\mu(x)<\i \quad \text{and} \quad \int_Md^r(x,x_0)\,d\nu(x)<\i
$$
for some given $x_0\in M$, then we can find a transport map $T:M\to M$, with $T_\sharp\mu=\nu$, which is optimal for the cost $d^r$ on $M\times M$. Moreover, the map $T$ is    uniquely determined $\mu$-almost everywhere.}
\end{theorem}
We recall that a measure on a smooth manifold is absolutely continuous with respect to the Lebesgue measure if, when one looks at it in charts, it is absolutely continuous with respect to Lebesgue measure. Again we note that there is no restriction on the curvature of $M$ in the theorem above.

The paper is structured as follows:
in section \ref{section2} we recall some known results on the general theory of the optimal transport problem, and we introduce some useful definitions.
Then, in section \ref{sectionmain} we will give very general results for the existence and the uniqueness of optimal transport maps
(Theorems \ref{keythm} and \ref{mainthm}, and Complement \ref{compmainthm}).
In section \ref{sectLagrangian} the above results are applied in the case of costs functions coming from (weak) Tonelli Lagrangians
(Theorems \ref{thLagragiancost} and \ref{coutRiemannien2}).
In section \ref{sectinterp}, we study the so called ``dispacement interpolation'', showing a countably
Lipschitz regularity for the transport map starting from an intermidiate time (Theorem \ref{interpolation}).
Finally, in the appendix, we collect all the tecnical results about semi-concave functions and Tonelli Lagrangians used in our proofs.

\emph{Acknowledgements:} The authors would like to thank
particularly C\'edric Villani for communicating to them his
Saint-Flour Lecture Notes \cite{villani2}. They also warmly thank
the anonymous referee for his comments and for pointing out to
them Remark \ref{rmkreferee}. The first author learned most of the
subject from invaluable informal almost daily conversations with
C\'edric Villani.

\section{Background and some definitions}\label{section2}
A major advance to solve the Monge transport problem is due to Kantorovich. He proposed in \cite{kant1}, \cite{kant2} a notion of weak solution of the
transport problem. He suggested to look for \textit{plans} instead of transport maps, that is probability measures $\g$ in $X \times Y$
whose marginals are $\mu$ and $\nu$, i.e.
$$
(\pi_X)_\sharp \gamma=\mu \quad \text{and} \quad (\pi_Y)_\sharp \gamma=\nu,
$$
where $\pi_X: X \times Y \rightarrow X$ and $\pi_Y: X \times Y \rightarrow Y$ are the canonical projections.
Denoting by $\Pi(\mu,\nu)$ the set of plans, the new minimization problem becomes then the following:

\begin{equation}
\label{kantprob}
C(\mu,\nu) =\min_{\g \in \Pi(\mu,\nu)} \left\{ \int_{M \times M} c(x,y)\,d\gamma(x,y) \right\}.
\end{equation}
If $\gamma$ is a minimizer for the Kantorovich formulation, we say that it is an \textit{optimal plan}.
Due to the linearity of the constraint $\g \in \Pi(\mu,\nu)$, it turns out that weak topologies can be used to provide existence
of solutions to (\ref{kantprob}): this happens for instance whenever $X$ and $Y$ are Polish spaces and $c$ is lower semicontinuous
(see \cite{racrus}, \cite[Proposition 2.1]{villani1} or \cite{villani2}). The connection between the formulation of Kantorovich and that of Monge can be seen by noticing that any transport
map $T$ induces the plan defined by $(\operatorname{Id}_X \tilde\times T)_\sharp \mu$ which is concentrated on the graph of $T$, where the map $\operatorname{Id}_X \tilde\times T:X\to X\times Y$ is defined by
$$\operatorname{Id}_X \tilde\times T(x)=(x,T(x)).$$

It is well-known that a linear minimization problem with convex constraints, like (\ref{kantprob}), admits a dual formulation.
Before stating the duality formula, we make some definitions similar to that of the weak KAM theory (see \cite{fathi}):

\begin{definition}[\bf $c$-subsolution]{\rm
We say that a pair of functions $\varphi:X \rightarrow \R \cup \{+\i\}$, $\psi:Y \rightarrow \R \cup \{-\i\}$
is a \textit{$c$-subsolution} if
$$
\forall (x,y) \in X \times Y, \quad  \psi(y)-\varphi(x) \leq c(x,y).
$$}
\end{definition}
Observe that when $c$ is measurable and bounded below, and  $(\varphi,\psi)$ is a $c$-subsolution with $\varphi\in L^1(\mu),\psi\in L^1(\nu)$, then
\begin{align*}
\forall \gamma\in \Pi(\mu,\nu), \quad \int_Y \psi \,d\nu-\int_X \varphi \,d\mu&=\int_{X \times Y} (\psi(y)-\varphi(x)) \,d\gamma(x,y)\\
&\leq \int_{X \times Y}c(x,y) \,d\gamma(x,y).
\end{align*}
If moreover $\int_{X \times Y}c(x,y) \,d\gamma <+\i$, and
$$
\int_{X \times Y} (\psi(y)-\varphi(x)) \,d\gamma(x,y) = \int_{X \times Y}c(x,y) \,d\gamma(x,y),
$$
then one would obtain the following equality:
$$
\psi(y)-\varphi(x) = c(x,y) \quad \text{for  $\gamma$-a.e. $( x,y)$}
$$
(without any measurability or integrability assumptions on $(\varphi,\psi)$, this is just a formal
computation).

\begin{definition}[\bf Calibration]{\rm
Given an optimal plan $\gamma$, we say that a $c$-subsolution $(\varphi,\psi)$
is \textit{$(c,\gamma)$-calibrated} if $\varphi$ and $\psi$ are Borel measurable, and
$$
\psi(y)-\varphi(x) = c(x,y) \quad \text{for  $\gamma$-a.e. $( x,y)$}.
$$}
\end{definition}

\begin{theorem}[\bf Duality formula]
\label{thmduality}{\sl
Let $X$ and $Y$ be Polish spaces equipped with probability measures $\mu$ and $\nu$ respectively,
$c: X \times Y \rightarrow \R$ a lower semicontinuous cost function bounded from below such that
the infimum in the Kantorovitch problem (\ref{kantprob}) is finite.
Then
a transport plan $\g \in \Pi(\mu,\nu)$ is optimal if and only if
there exists a $(c,\gamma)$-calibrated subsolution $(\varphi,\psi)$.}
\end{theorem}
For a proof of this theorem see \cite{schteich} and \cite[Theorem 5.9 (ii)]{villani2}.\par
In this work we  study Monge's problem on manifolds for a large class of cost functions induced by Lagrangians like in \cite{bernbuf}, where the authors consider the case of compact manifolds. We  generalize their result to arbitrary non-compact manifolds.

Following the general scheme of proof, we will first prove a result on more general costs, see Theorem \ref{mainthm}. In this general result, the fact that the target space for the Monge transport is a manifold is not necessary. So we will assume that only
the source space (for the Monge transport map) is a manifold.

Let $M$ be a $n$-dimensional manifold (Hausdorff and with a countable basis),
$N$ a Polish space, $c: M \times N \rightarrow \R$ a cost function,
$\mu$ and $\nu$ two probability measures on $M$ and $N$ respectively.
We want to prove existence and uniqueness of an  optimal transport map $T : M \rightarrow N$, under some reasonable hypotheses on $c$ and $\mu$.

One of the conditions on the cost  $c$  is given in the following definition:
\begin{definition}[Twist Condition]\label{twistcondition}{\rm
For a given cost function $c(x,y)$, we define the \textit{skew left Legendre transform} as the partial map
$$
\Lambda_c^l: M \times N \rightarrow T^*M,
$$
$$
\Lambda_c^l(x,y)=(x,\frac{\partial c}{\partial x}(x,y)),
$$
whose domain of definition is
$$
\mathcal D(\Lambda_c^l)=\left\{ (x,y) \in M \times N\mid \frac{\partial c}{\partial x}(x,y) \mbox{ exists} \right\}.
$$
Moreover, we say that $c$ satisfies the \textit{left twist condition} if
$\Lambda_c^l$ is injective on $\mathcal D(\Lambda_c^l)$.

One can define similarly the \textit{skew right Legendre transform} $
\Lambda_c^r: M \times N \rightarrow T^*N$ by $
\Lambda_c^r(x,y)=(y,\frac{\partial c}{\partial y}(x,y)),
$. The domain of definition of $\Lambda_c^r$ is $
\mathcal D(\Lambda_c^r)=\{ (x,y) \in M \times N\mid \frac{\partial c}{\partial x}(x,y) \mbox{ exists}\}$. We say that $c$ satisfies the \textit{right twist condition} if
$\Lambda_c^r$ is injective on $\mathcal D(\Lambda_c^r)$.}
\end{definition}

The usefulness of these definitions will be clear in the section \ref{sectLagrangian}, in which we will
treat the case where $M=N$ and the cost is induced by a Lagrangian. This condition has appeared already in the subject. It has been known (explicitly or not) by several people, among them Gangbo (oral communication) and Villani (see \cite[page 90]{villani1}). It is used in \cite{bernbuf},
since it is always satisfied for a cost coming from a Lagrangian, as we will see below. We borrow the terminology ``twist condition" from the theory of Dynamical Systems: if $h:\R\times \R\to \R, (x,y)\mapsto h(x,y)$ is C$^2$, one says that $h$ satisfies the twist condition if there exists a constant $\alpha>0$ such that $\displaystyle\frac{\partial^2h}{\partial x\partial y} \geq \alpha$ everywhere. In that case both maps $\Lambda^l_h:\R\times\R\to \R\times\R, (x,y)\mapsto (x,\partial h/\partial x(x,y))$ and  $\Lambda^r_h:\R\times\R\to \R\times\R, (x,y)\mapsto (y,\partial h/\partial y(x,y))$ are C$^1$ diffeomorphisms. The twist map $f:\R\times\R\to \R\times\R$ associated to $h$ is determined by $f(x_1,v_1)=(x_2,v_2)$, where
$v_1=-\partial h/\partial x(x_1,x_2), v_2=\partial h/\partial y(x_1,x_2)$, which means
$f(x_1,v_1)= \Lambda^r_h\circ [\Lambda^l_h]^{-1}(x_1,-v_1)$, see \cite{mather} or
\cite{forni-mather}.

We now recall some useful measure-theoretical facts that we will need in the sequel.
\begin{lemma}
\label{lemma meas}
{\sl Let $M$ be an $n$-dimensional manifold, $N$ be a Polish space, and let $c:M \times N \to \R$ be a measurable function such that
$x \mapsto c(x,y)$ is continuous for any $y \in N$. Then the set
$$
\{ (x,y) \mid \frac{\p c}{\p x}(x,y) \ \text{exists} \} \quad \text{is Borel measurable}.
$$
Moreover $(x,y) \mapsto \frac{\p c}{\p x}(x,y)$ is a Borel function on that set.}
\end{lemma}
\begin{proof} This a standard result in measure theory, we give here just a sketch of the proof.

By the locality of the statement, using charts we can assume $M=\R^n$.
Let $T_k:\R^n \to \R^n$ be a dense countable family of linear maps.
For any $j,k \in \N$, we consider the Borel function
\begin{align*}
L_{j,k}(x,y):&=\sup_{|h| \in (0,\frac{1}{j})} \frac{|c(x+h,y)-c(x,y) - T_k(h)|}{|h|}\\
&=\sup_{|h| \in (0,\frac{1}{j}) , h\in \Q^n} \frac{|c(x+h,y)-c(x,y) - T_k(h)|}{|h|},
\end{align*}
where in the second equality we used the continuity of $x \mapsto c(x,y)$.
Then it is not difficult to show that the set of point where $\frac{\p c}{\p x}(x,y)$ exists can be written as
$$
\{ (x,y) \mid \inf_j \inf_k L_{j,k}(x,y)=0\},
$$
which is clearly a Borel set.

To show that $x \mapsto \frac{\p c}{\p x}(x,y)$ is Borel,
it suffices to note that the partial derivatives
$$
\frac{\partial c}{\partial x_i}(x,y)=\lim_{\ell \to\infty}
\frac{c(x_1,\dots,x_i+\frac{1}{\ell},\dots,x_n,y)-c(x_1,\dots,x_i,\dots,x_n,y)}{1/\ell}
$$
are countable limits of continuous functions, and hence are Borel measurable.
\end{proof}

Therefore, by the above lemma, $\mathcal D(\Lambda_c^l)$ is a Borel set.
If we moreover assume that $c$ satisfies the left twist condition (that is, $\Lambda_c^l$ is injective on $\mathcal D(\Lambda_c^l)$),
then one can define
$$
(\Lambda_c^l)^{-1}:T^*M \supset \Lambda_c^l(\mathcal D(\Lambda_c^l)) \to \mathcal D(\Lambda_c^l) \subset M \times N.
$$
Then, by the injectivity assumption, one has that $\Lambda_c^l(\mathcal D(\Lambda_c^l))$ is still a Borel set,
and $(\Lambda_c^l)^{-1}$ is a Borel map (see \cite[Proposition  8.3.5 and Theorem  8.3.7]{cohn}, \cite{federer}).
We can so extend $(\Lambda_c^l)^{-1}$ as a Borel map on the whole $T^*M$ as
$$
\Lambda_c^{l,inv}(x,p) =\left\{
\begin{array}{ll}
(\Lambda_c^l)^{-1}(x,p) & \text{if } p \in T_x^*M \cap \Lambda_c^l(\mathcal D(\Lambda_c^l)),\\
(x,\bar y) & \text{if } p \in T_x^*M \setminus \Lambda_c^l(\mathcal D(\Lambda_c^l)),
\end{array}
\right.
$$
where $\bar y$ is an arbitrary point, but fixed point, in $N$.

\section{The main result}
\label{sectionmain}
In order to have general results of existence and uniqueness of transport maps
which are sufficiently flexible so that they can be used also in other situations,
and to well show where measure-theoretic problems enter in the proof of the existence of the transport map,
we will first give a general result where no measures are present
(see in the appendix \ref{defsemicon} for the definition of locally semi-concave function and
\ref{defcountLip} for the definition of countably $(n-1)$-Lipschitz set).
\begin{theorem}
\label{keythm}{\sl Let $M$ be a smooth (second countable) manifold, and let $N$ be a Polish space.
Assume that the cost $c:M\times N\to\R$ is Borel measurable, bounded from below, and satisfies the following conditions:

\begin{enumerate}
\item[{\rm (i)}]
the family of maps $x \mapsto c(x,y)=c_y(x)$ is locally semi-concave in $x$ locally uniformly in $y$,
\item[{\rm (ii)}]
the cost $c$ satisfies the left twist condition.
\end{enumerate}
Let $(\varphi,\psi)$ be a $c$-subsolution, and consider the set $G_{(\varphi,\psi)} \subset M \times N$ given by
$$
G_{(\varphi,\psi)} =\{(x,y) \in M \times N \mid \psi(y)-\varphi (x)=c(x,y) \}.
$$
We can find  a Borel countably $(n-1)$-Lipschitz set $E \subset M$ and a Borel measurable map $T:M\to N$ such that
$$G_{(\varphi,\psi)} \subset \graph(T)\cup \pi_M^{-1}(E),$$
where  $\pi_M :M \times N \to M$ is the canonical projection, and $\graph(T)=\{(x,T(x))\mid x\in M\}$ is the graph of $T$.

In other words, if we define $P=\pi_M\bigl(G_{(\varphi,\psi)} \bigr)\subset M$
the part of
$G_{(\varphi,\psi)}$ which is above $P \setminus E$  is contained a Borel graph.

More precisely,
we will prove that there exist an increasing sequence of
locally semi-convex functions $\varphi_n:M\to\R$, with $\varphi\geq\varphi_{n+1}\geq\varphi_n$ on $M$, and an increasing sequence of Borel subsets $C_n$ such that
\begin{itemize}
\item For $x\in C_n$,
the derivative $d_x\varphi_n$ exists, $\varphi_{n+1}(x)=\varphi_n(x)$, and $d_x\varphi_{n+1}=d_x\varphi_n$.
\item If we set $C =\cup_n C_n$, there exists a Borel countably $(n-1)$-Lipschitz set $E \subset M$
such that $P \setminus E \subset C$.
\end{itemize}
Moreover, the Borel map $T:M\to N$ is such that
\begin{itemize}
\item For every $x\in C_n$, we have
$$(x,T(x))=\Lambda^{l,inv}_c(x,-d_x\varphi_n),$$
where $\Lambda^{l,inv}_c$ is the extension of the inverse of $\Lambda^{l}_c$ defined at the end of section \ref{section2}.
\item If $x\in P\cap C_n\setminus E$, then
the partial derivative $\displaystyle \frac{\partial c}{\partial x}(x,T(x))$ exists  (i.e. $(x,T(x))\in \mathcal D(\Lambda_c^l)$ ), and
$$
\frac{\partial c}{\partial x}(x,T(x))=-d_x\varphi_n.
$$
\end{itemize}
In particular, if $x\in  P\cap C_n\setminus E$, we have
$$
(x,T(x))\in \mathcal D(\Lambda_c^l)\quad
\text{and}\quad
\Lambda^l_c(x,T(x))= (x-d_x\varphi_n).
$$
Therefore, thanks to the twist condition, the map $T$ is uniquely defined on $P\setminus E \subset C$.}
\end{theorem}

The existence and uniqueness of a transport map is then a simple consequence of the above theorem.
\begin{theorem}
\label{mainthm}{\sl Let $M$ be a smooth (second countable) manifold, let $N$ be a Polish space,
and consider $\mu$ and $\nu$ (Borel) probability measures on $M$ and $N$ respectively.
Assume that the cost $c:M\times N\to\R$ is lower semicontinuous and bounded from below.
Assume moreover that the following conditions hold:

\begin{enumerate}
\item[{\rm (i)}]
the family of maps $x \mapsto c(x,y)=c_y(x)$ is locally semi-concave in $x$ locally uniformly in $y$,
\item[{\rm (ii)}]
the cost $c$ satisfies the left twist condition,
\item[{\rm (iii)}]
the measure $\mu$ gives zero mass to countably $(n-1)$-Lipschitz sets,
\item[{\rm (iv)}]
the infimum in the Kantorovitch problem (\ref{kantprob}) is finite.
\end{enumerate}
Then there exists a Borel map $T:M\to N$, which is an optimal transport map from $\mu$ to $\nu$ for the cost $c$. Morover, the map $T$ is unique $\mu$-a.e., and  any plan $\gamma_c\in\Pi(\mu,\nu)$ optimal for the cost $c$ is concentrated on the graph of $T$.

More precisely, if $(\varphi,\psi)$ is a $(c,\g_c)$-calibrating pair,
with the notation of Theorem \ref{keythm},
there exists an increasing sequence of Borel subsets $B_n$, with $\mu(\cup_n B_n)=1$,
such that the map $T$ is uniquely defined on $B =\cup_n B_n$ via
$$
\frac{\partial c}{\partial x}(x,T(x))=-d_x\varphi_n \quad \text{on }B_n,
$$
and any optimal plan $\gamma\in\Pi(\mu,\nu)$ is concentrated on the graph of that map $T$.}
\end{theorem}

We remark that condition (iv) is trivially satisfied if
$$
\int_{M \times N} c(x,y) \,d\mu(x) \,d\nu(y) <+\infty.
$$
However we needed to stated the above theorem in this more general form in order to apply it in section \ref{sectinterp}
(see Remark \ref{rmkintegrability}).

\begin{proof}[Proof of Theorem \ref{mainthm}]
Let $\gamma_c\in\Pi(\mu,\nu)$ be an optimal plan. By Theorem \ref{thmduality} there exists a
$(c,\gamma)$-calibrated pair $(\varphi,\psi)$.
Consider the set
$$
G =G_{(\varphi,\psi)}=\{(x,y) \in M \times N \mid \psi(y)-\varphi(x)=c(x,y) \}.
$$
Since both $M$ and $N$ are Polish and both maps $\varphi$ and $\psi$ are Borel, the subset
$G$ is a Borel subset of $M \times N$.
Observe that, by the definition of $(c,\gamma_c)$-calibrated pair, we have $\gamma_c(G)=1$.

By Theorem \ref{keythm} there exists a Borel countably $(n-1)$-Lipschitz set $E$ such that
$G \setminus (\pi_M)^{-1}(E)$ is contained in the graph of a Borel map $T$.
This implies that
$$
B =\pi_M \left(G \setminus (\pi_M)^{-1}(E)\right)=\pi_M(G)\setminus E\subset M
$$
is a Borel set, since it coincides with $(\Id_M\tilde \times T)^{-1}\left(G \setminus (\pi_M)^{-1}(E) \right)$ and
the map $x\mapsto\Id_M\tilde \times T(x)=(x,T(x))$ is Borel measurable.

Thus, recalling that the first marginal of $\g_c$ is $\mu$, by assumption (iii) we get $\gamma_c((\pi_M)^{-1}(E))=\mu(E)=0$.
Therefore $\gamma_c(G \setminus (\pi_M)^{-1}(E))=1$, so that $\g_c$ is concentrated on the graph of $T$,
which gives the existence of an optimal transport map.
Note now that $\mu(B)=\gamma_c(\pi^{-1}(B))\geq\gamma_c(G \setminus (\pi_M)^{-1}(E))=1$. Therefore $\mu(B)=1$. Since $B=P\setminus E$, where $P=\pi_M(G)$, using the Borel set $C_n$ provided by Theorem \ref{keythm}, it follows
that $B_n=P\cap C_n\setminus E=D\cap C_n$ is a Borel set with $B=\cup_n B_n$.
The end of Theorem \ref{keythm} shows that $T$ is indeed uniquely defined on $B$ as said in the statement.

Let us now prove the uniqueness of the transport map $\mu$-a.e..
If $S$ is another optimal transport map, consider the measures $\g_T =(\operatorname{Id}_M \times T)_\# \mu$
and $\g_S =(\operatorname{Id}_M \times S)_\# \mu$. The measure
$\bar \g = \frac{1}{2}(\g_T + \g_S)\in\Pi(\mu,\nu)$ is still an optimal plan, and therefore must be concentrated on a graph.
This implies for instance that $S=T$ $\mu$-a.e., and thus $T$ is the unique optimal transport map.
Finally, since any optimal $\g \in \Pi(\mu,\nu)$ is concentrated on a graph,
we also deduce that any optimal plan is concentrated on the graph of $T$.
\end{proof}

\begin{proof}[Proof of Theorem \ref{keythm}]
By definition of $c$-subsolution, we have $\varphi>-\infty$ everywhere on $M$, and $\psi<+\infty$ everywhere on $N$.
Therefore, if we define $W_n =\{ \psi \leq n\}$, we have $W_n\subset W_{n+1}$, and $\cup_n W_n=N$.
Since, by hypothesis {(i)}, $c(x,y)=c_y(x)$ is locally semi-concave in $x$ locally uniformly in $y$, for each $y \in N$
there exist a neighborhood $V_y$ of $y$ such that the family of functions $(c(\cdot,z))_{z \in V_y}$ is locally
uniformly semi-concave. Since $N$ is separable, there exists a countable family of points $(y_k)_{k \in \N}$
such that $\cup_k V_{y_k}=N$.
We now consider the sequence of subsets $(V_n)_{n \in \N} \subset N$ defined as
$$
V_n =W_n \cap \left(\cup_{1 \leq k \leq n} V_{y_k} \right).
$$
We have $V_n\subset V_{n+1}$.
Define $\varphi_n:M\to N$ by
$$
\varphi_n(x) =\sup_{y \in V_n} \psi(y) - c(x,y)
=\max_{1 \leq k \leq n} \Bigl( \sup_{y \in W_n \cap V_{y_k}} \psi(y) - c(x,y) \Bigr).
$$
Since $\psi \leq n$ on $K_n$, and $-c$ is bounded from above, we see that $\varphi_n$ is bounded from above.
Therefore, by hypothesis {(i)}, the family of functions  $(\psi(y)-c(\cdot,y))_{y\in W_n \cap V_{y_k}}$ is locally
uniformly semi-convex and bounded from above. Thus, by Theorem \ref{EasyPropSemiConcave} and Proposition \ref{InfSemiConcave}
of the Appendix,
the function $\varphi_n$ is locally semi-convex. Since $\psi(y)-\varphi(x)\leq c(x,y)$ with equality on
$G_{(\varphi,\psi)}$, and $V_n\subset V_{n+1}$, we clearly have
$$
\varphi_n \leq \varphi_{n+1}\leq\varphi \quad \text{everywhere on $M$}.
$$
A key observation is now the following:
$$
\varphi|_{_{P_n}}={\varphi_n}|_{_{P_n}},
$$
where $P_n=\pi_M(G_{(\varphi,\psi)} \cap (M\times V_n))$. In fact, if $x \in P_n$, by the definition of $P_n$ we know that there exists a point
$y_x \in V_n$ such that $(x,y_x) \in G_{(\varphi,\psi)}$. By the definition of $G_{(\varphi,\psi)}$, this implies
$$
\varphi(x)=\psi(y_x)-c(x,y_x) \leq \varphi_n(x) \leq \varphi(x).
$$
Since $\varphi_n$ is locally semi-convex, by Theorem \ref{propdiffsemiconc} of the appendix applied to $-\varphi_n$,
it is differentiable on a Borel subset $F_n$ such that its complement $F^c_n$ is a Borel countably $(n-1)$-Lipschitz set.
Let us then define $F =\cap_n F_n$.
The  complement $E =F^c=\cup_n F_n^c$ is also a Borel countably $(n-1)$-Lipschitz set.
We now define the Borel set
$$
C_n =F \cap \{x \in M \mid \varphi_k(x)=\varphi_n(x) \ \ \forall k \geq n \}.
$$
We observe that $C_n \supset P_n \cap F$.
\par
We now prove that
$G_{(\varphi,\psi)} \cap \left((P_n\cap F) \times V_n \right)$ is contained in a graph.\par
To prove this assertion, fix $x \in P_n \cap F$. By the definition of $P_n$, and what we said above,
there exists $y_x \in V_n$ such that
$$
\varphi(x)=\varphi_n(x)=\psi(y_x)-c(x,y_x).
$$
Since $x \in F$, the map
$
z \mapsto \varphi_n(z) - \psi(y_x)
$
is differentiable at $x$. Moreover, by condition (i), the map
$
z \mapsto -c(z,y_x)=-c_{y_x}(z)$ is locally semi-convex
and, by the definition of $\varphi_n$, for every  $z \in M$, we have
$\varphi_n(z)-\psi(y_x)\geq -c(z,y_x)$, with equality at $z=x$.
These facts taken together imply that $\displaystyle\frac{\partial c}{\partial x}(x,y_x)$ exists and is equal to $-d_x\varphi_n$.
In fact, working in a chart around $x$, since $c_{y_x}=c(\cdot,y_x)$ is locally semi-concave, by the definition
\ref{defsemicon} of a locally semi-concave function, there exists linear map $l_x$ such that
$$
c(z,y_x) \leq c(x,y_x) + l_x(z-x) + o(|z-x|),
$$
for $z$ in a neighborhood of $x$. Using also that $\varphi_n$ is differentiable at  $x$, we get
\begin{align*}
\varphi_n(x)-\psi(y_x) &+ d_x\varphi_n(z-x) + o(|z-x|)=
\varphi_n(z)-\psi(y_x)\\
& \geq -c(z,y_x)
\geq -c(x,y_x) - l_c(z-x) + o(|z-x|)\\
&=\varphi_n(x)-\psi(y_x) - l_c(z-x)+ o(|z-x|).
\end{align*}
This implies that $l_c=-d_x\varphi_n$, and that $c_{y_x}$ is differentiable at $x$ with differential at $x$ equal to $-d_x\varphi_n$.
Setting now $G_x = \{ y \in N \mid \varphi(x)-\psi(y)=c(x,y)\}$, we have just shown that
$\{x\} \times (G_x \cap V_n ) \subset \mathcal D(\Lambda_c^l)$ for each $x \in C_n$,
and also $\displaystyle\frac{\partial c}{\partial x}(x,y)=-d_x\varphi_n$, for every $y\in G_x \cap V_n$.
Recalling now that, by hypothesis (ii), the cost $c$ satisfies the left twist condition,
we obtain that $G_x\cap V_n $ is reduced to a single element
which is uniquely characterized by the equality
$$
\frac{\partial c}{\partial x}(x,y_x)=-d_x\varphi_n.
$$
So we have proved that $G \cap (M \times V_n)$ is the graph over $P_n \cap F$ of the map $T$ defined uniquely,
thanks to the left twist condition, by
$$
\frac{\partial c}{\partial x}(x,T(x))=-d_x\varphi_n
$$
(observe that, since $\varphi_n \leq \varphi_k$ for $k \geq n$ with equality on $P_n$,
we have $d_x\varphi_n|_{P_n}=d_x\varphi_k|_{P_n}$ for $k \geq n$).
Since $P_{n+1}\supset P_n$, and $V_n\subset V_{n+1} \nearrow N$, we can conclude that
$G_{(\varphi,\psi)}$ is a graph over $\cup_n P_n \cap F=P \cap F$ (where $P=\pi_M(G_{(\varphi,\psi)})=\cup_n P_n$).

Observe that, at the moment, we do not know that $T$ is a Borel map, since $P_n$ is not a priori Borel.
Note first that by definition of $B_n\subset B_{n+1}$, we have $\varphi_n=\varphi_{n+1}$ on $B_n$,
and they are both differentiable at every point of $B_n$. Since
$\varphi_n\leq \varphi_{n+1}$ everywhere, by the same argument as above we get
$d_x\varphi_n=d_x \varphi_{n+1}$  for $x\in B_n$. Thus, setting $B =\cup_n B_n$, we can extend $T$ to $M$ by
$$
T(x) =\left\{
\begin{array}{ll}
\pi_N \Lambda_c^{l,inv}(x,-d_x\varphi_n) & \text{on }B_n,\\
\bar y & \text{on }M \setminus B,
\end{array}
\right.
$$
where $\pi_N:M \times N \to N$ is the canonical projection,
$ \Lambda_c^{l,inv}$ is the Borel extension of $(\Lambda_c^l)^{-1}$
defined after Lemma \ref{lemma meas},
and $\bar y$ is an arbitrary but fixed point in $N$. Obviously, the map $T$ thus defined
is Borel measurable and extends the map $T$ already defined on $P\setminus E$.
\end{proof}

In the case where $\mu$ is absolutely continuous with respect to Lebesgue measure we can give a complement to our main theorem. In order to state it, we need the following definition, see \cite[Definition 5.5.1, page 129]{amgisa}:
\begin{definition}[\bf Approximate differential]\label{approxdiff}{\rm
We say that $f :M \rightarrow \R$ has an \textit{approximate differential} at $x \in M$
if there exists a function $h:M \rightarrow \R$
differentiable at $x$ such that the set $\{f = h\}$ has density $1$ at $x$ with respect to the Lebesgue measure
(this just means that the density is $1$ in charts).
In this case, the approximate value of $f$ at $x$ is defined as $\tilde f(x)=h(x)$, and the approximate differential of $f$ at $x$ is defined as $\tilde d_xf=d_xh$.
It is not difficult to show that this definition makes sense. In fact, both $h(x)$, and $d_xh$ do not depend on the choice of $h$, provided $x$ is a density point of the set $\{f = h\}$.}
\end{definition}
Another characterization of the approximate value $\tilde f(x)$, and  of the approximate differential $\tilde d_xf$ is given, in charts, saying that the sets
$$
\left\{ y \mid \frac{|f(y)-\tilde f(x) -\tilde d_xf(y-x)|}{|y-x|} > \varepsilon \right\}
$$
have density $0$ at $x$ for each $\varepsilon >0$ with respect to Lebesgue measure. This last definition
is the one systematically used in \cite{federer}. On the other hand, for the purpose of this paper, Definition \ref{approxdiff}
is more convenient.

The set points $x\in M$ where the approximate derivative $\tilde d_xf$ exists is measurable;
moreover, the map $x\mapsto \tilde d_xf$ is also measurable, see \cite[Theorem 3.1.4, page 214]{federer}.

\begin{ajout}\label{compmainthm}{\sl Under the hypothesis of Theorem \ref{mainthm}, if we assume that $\mu$ is absolutely continuous with respect to Lebesgue measure (this is stronger than condition (iii) of Theorem \ref{mainthm}), then for any calibrated pair $(\varphi,\psi)$, the function $\varphi$ is approximatively differentiable $\mu$-a.e., and the optimal transport map $T$ is uniquely determined $\mu$-a.e., thanks to the twist condition, by
$$
\frac{\partial c}{\partial x}(x,T(x))=-\tilde d_x\varphi,
$$
where $\tilde d_x\varphi$ is the approximate differential of $\varphi$ at $x$.
Moreover, there exists a Borel subset $A\subset M$ of full $\mu$ measure
such that $\tilde d_x\varphi$ exists on $A$,
the map $x\mapsto \tilde d_x\varphi$ is Borel measurable on $A$,
and $\displaystyle \frac{\partial c}{\partial x}(x,T(x))$ exists for
$x\in A$ (i.e. $(x,T(x))\in \mathcal D(\Lambda_c^l)$).}
\end{ajout}
\begin{proof} We will use the notations and the proof of Theorems \ref{keythm} and \ref{mainthm}.
We denote by $\tilde A_n\subset B_n$ the set of $x\in B_n$ which are density points for $B_n$ with respect to some measure $\lambda$ whose measure class in charts is that of Lebesgue (for example one can take $\lambda$ as the Riemannian measure associated to a Riemannian metric). By Lebesgue's Density Theorem $\lambda(B_n\setminus\tilde A_n)=0$. Since $\mu$ is absolutely continuous with respect to Lebesgue measure, we have
$\mu(\tilde A_n)=\mu(B_n)$, and therefore $\tilde A=\cup_n \tilde A_n$ is of full $\mu$-measure,
since $\mu(B_n)\nearrow \mu(B)= 1$. Moreover, since $\varphi=\varphi_n$ on $B_n$, and $\varphi_n$ is differentiable at each point of $B_n$, the function $\varphi$ is approximatively differentiable at each point of $\tilde A_n$ with $\tilde d_x\varphi=d_x\varphi_n$.

The last part of this complement on measurability follows of course from \cite[Theorem 3.1.4, page 214]{federer}. But in this case, we can give a direct simple proof. We choose
$A_n\subset \tilde A_n$ Borel measurable with $\mu(\tilde A_n\setminus A_n)=0$. We set $A=\cup_nA_n$. The set $A$ is of full $\mu$ measure. Moreover,
for every $x\in A_n$, the approximate differential $\tilde d_x\varphi$ exists and is equal to $d_x\varphi_n$.
Thus it suffices to show that the map $x\mapsto d_x\varphi_n$ is Borel measurable,
and this follows as in Lemma \ref{lemma meas}.
\end{proof}

\begin{remark}
\label{rmkreferee}{\rm We observe that the key steps in the
results of this section, translated in a bit different language,
can be summarized in the following way. Let $\varphi:M \to \R \cup
\{+\infty\}$ be a $c$-convex function, that is
$$
\varphi(x)=\sup_{y \in N} \psi(y)-c(x,y)
$$
for a certain $\psi:N \to \R \cup \{-\i\}$, and assume $\varphi
\not\equiv +\i$. Then, if $c$ satisfies hypothesis (i) of Theorem
\ref{keythm}, we can prove the following:
\begin{enumerate}
\item[(1)] There exists a non-decreasing sequence $\varphi_n$ of
locally semi-convex functions such that
$\cup_{n}\{\varphi=\varphi_n\}=M$.
\end{enumerate}
From this we deduce:
\begin{enumerate}
\item[(2a)] There exists a Borel section $P(x)$ of the cotangent
bundle $T^*M$ and a $(n-1)$-Lipschitz set $E \subset M$ such that,
if $x \not\in E$, $P(x)$ is the only possible sub-differential of
$\varphi$ at $x$.

\item[(2b)] The function $\varphi$ is approximately differentiable
a.e. with respect to the Lebesgue measure.
\end{enumerate}
By these facts, the duality theorem and the semi-concavity of the
cost, we can deduce the existence and uniqueness of an optimal
transport (under the general assumptions of Theorem \ref{mainthm}
we use (2a), while under the assumption that the source measure is
absolutely continuous with respect to the Lebesgue measure we use
(2b)).
 }
\end{remark}

\section{Costs obtained from Lagrangians}\label{sectLagrangian}
Now that we have proved Theorem \ref{mainthm}, we want to observe that the hypotheses are satisfied by a large class of cost functions.

We will consider first the case of a Tonelli Lagrangian $L$ on a connected manifold (see Definition \ref{TonelliLagrangian} of the appendix for the definition of a Tonelli Lagrangian). For $t>0$, the cost $c_{t,L}:M\times M\to \R$ associated to $L$ is given by
$$
c_{t,L}(x,y)=\inf_{\gamma} \int_0^t L(\gamma(s),\dot \gamma(s)) \,ds,
$$
where the infimum is taken over all the continuous piecewise C$^1$ curves
$\gamma:[0,t]\to M$, with $\gamma(0)=x$, and $\gamma(t)=y$ (see
Definition \ref{CostLagrangian} of the appendix).
\begin{proposition}\label{C2LagAreOK}{ \sl  If $L:TM \rightarrow \R$ is a Tonelli Lagrangian on the connected manifold $M$, then, for $t>0$, the cost $c_{t,L}:M\times M\to \R$ associated to the Lagrangian $L$ is continuous, bounded from below, and satisfies conditions
{\rm (i)} and {\rm (ii)} of Theorem \ref{mainthm}.}
\end{proposition}
\begin{proof} Since $L$ is a Tonelli Lagrangian, observe that $L$ is bounded below by $C$, where $C$ is the constant given in condition (c) of Definition
\ref{TonelliLagrangian}. Hence the cost $c_{t,L}$ is  bounded below by $tC$. By Theorem \ref{costsemiconc} of the appendix, the cost $c_{t,L}$ is locally semi-concave, and therefore continuous. Moreover, we can now apply
Proposition \ref{uniformlocalsemiconcave} of the appendix to conclude that
$c_{t,L}$ satisfies condition
{\rm (i)} of Theorem \ref{mainthm}.

The twist condition {\rm (ii)} of Theorem \ref{mainthm} for $c_{t,L}$ follows from Lemma  \ref{lemmaUC'} and Proposition \ref{TwistForTonelli}.
\end{proof}
For costs coming from Tonelli Lagrangians, we subsume the application of the main Theorem \ref{mainthm}, and its Complement \ref{compmainthm}.
\begin{theorem}\label{thLagragiancost}{\sl Let $L$ be a Tonelli Lagrangian on the connected manifold $M$.
Fix $t>0$, $\mu,\nu$ a pair of probability measure on $M$, with $\mu$ giving measure zero to
countably $(n-1)$-Lipschitz sets, and
assume that
the infimum in the Kantorovitch problem (\ref{kantprob}) with cost $c_{t,L}$ is finite.
Then there exists a uniquely $\mu$-almost everywhere defined transport map $T:M\to M$ from $\mu$ to $\nu$
which is optimal for the cost $c_{t,L}$. Moreover, any plan $\gamma \in \Pi(\mu,\nu)$, which is optimal
for the cost $c_{t,L}$, verifies $\gamma(\operatorname{Graph}(T))=1$.

If $\mu$ is absolutely continuous with respect to Lebesgue measure, and $(\varphi,\psi)$ is a $c_{t,L}$-calibrated subsolution for $(\mu,\nu)$, then we can find a Borel set $B$ of full $\mu$ measure, such that the  approximate differential
$\tilde d_x\varphi$  of $\varphi$ at $x$ is defined for $x\in B$, the map $x\mapsto \tilde d_x\varphi$
is Borel measurable on $B$,
and the transport map $T$ is defined on $B$ (hence $\mu$-almost everywhere) by
$$
T(x)=\pi^*\phi^H_t(x,\tilde d_x\varphi),
$$
where $\pi^*:T^*M\to M$ is the canonical projection, and $\phi^H_t$ is the Hamiltonian flow of the Hamiltonian $H$ associated to $L$.

We can also give the following description for $T$ valid on $B$  (hence $\mu$-almost everywhere):
$$T(x)=\pi\phi^L_t(x,\widetilde\grad^L_x(\varphi)),$$
where $\phi^L_t$ is the Euler-Lagrange flow of $L$, and $x\to\widetilde\grad^L_x(\varphi)$
is the measurable vector field on $M$ defined on $B$ by
$$
\frac{\partial L}{\partial v}(x,\widetilde\grad^L_x(\varphi))=\tilde d_x\varphi.
$$
Moreover, for  every $x\in B$, there is a unique $L$-minimizer $\gamma:[0,t]\to M$, with $\gamma(0)=x, \gamma(t)=T(x)$,
and this curve $\gamma$ is given by $\gamma(s)=\pi\phi^L_s(x,\widetilde\grad^L_x(\varphi))$, for $0\leq s\leq t$.}
\end{theorem}
\begin{proof} The first part is a consequence of Proposition \ref{C2LagAreOK} and Theorem \ref{mainthm}.
When $\mu$ is absolutely continuous with respect to Lebesgue measure, we can apply Complement \ref{compmainthm} to  obtain a Borel subset
$A\subset M$ of full $\mu$ measure such  that, for every $x\in A$, we have $(x,T(x))\in {\cal D}(\Lambda ^l_{c_{t,L}})$ and
$$
\frac{\partial c_{t,L}}{\partial x}(x,T(x))=\tilde d_x\varphi.
$$
By Lemma \ref{lemmaUC'} and Proposition \ref{TwistForTonelli},
if $(x,y)\in {\cal D}(\Lambda ^l_{c_{t,L}})$, then  there is a unique $L$-minimizer $\gamma:[0,t]\to M$, with $\gamma(0)=x, \gamma(t)=y$, and this minimizer is of the form
$\gamma(s)=\pi\phi^L_s(x,v)$, where $\pi:TM\to M$ is the canonical projection, and $v\in T_xM$ is uniquely determined by the equation
$$
\frac{\partial c_{t,L}}{\partial x}(x,y)=-\frac{\partial L}{\partial v}(x,v).
$$
Therefore $T(x)=\pi\phi^L_t(x,\widetilde\grad^L_x(\varphi))$, where $\widetilde\grad^L_x(\varphi)$ is uniquely determined by
$$
\frac{\partial L}{\partial v}(x,\widetilde\grad^L_x(\varphi))=-\frac{\partial c_{t,L}}{\partial x}(x,T(x))=\tilde d_x\varphi,
$$
which is precisely the second description of $T$. The first description of $T$ follows from the second one, once we observe that
\begin{align*}
\Leg(x,\widetilde\grad^L_x(\varphi)))&=(x,\frac{\partial L}{\partial v}(x,\widetilde\grad^L_x(\varphi))= (x,\tilde d_x\varphi)\\
\phi^H_t&=\Leg\circ \phi^L_t\circ \Leg^{-1}\\
\pi^*\circ\Leg&=\pi,
\end{align*}
where $\Leg:TM \to T^*M$ is the global Legendre Transform,see Definition \ref{defGlobLegTransf} of the appendix.
\end{proof}
We now turn to the proof of Theorem \ref{coutRiemannien}, which is not a consequence of Theorem \ref{thLagragiancost} since the cost
$d^r$ with $r>1$ does not come from a Tonelli Lagrangian for $r \neq 2$.

\begin{theorem}\label{coutRiemannien2}{\sl Suppose that the connected manifold $M$ is endowed with a Riemannian metric $g$
which is complete. Denote by $d$ the Riemannian distance. If $r>1$, and $\mu$ and $\nu$ are probability (Borel) measures on
$M$, where $\mu$ gives measure zero to countably $(n-1)$-Lipschitz sets, and
$$
\int_Md^r(x,x_0)\,d\mu(x)<\i \quad \text{and} \quad\int_Md^r(x,x_0)\,d\nu(x)<\i
$$
for some given $x_0\in M$, then we can find a transport map $T:M\to M$, with $T_\sharp\mu=\nu$, which is optimal for the cost $d^r$ on $M\times M$. Moreover, the map $T$ is uniquely determined $\mu$-almost everywhere.

If $\mu$ is absolutely continuous with respect to Lebesgue measure, and $(\varphi,\psi)$ is a calibrated subsolution for the cost $d^r(x,y)$ and the pair of measures $(\mu,\nu)$, then the  approximate differential $\tilde d_x\varphi$  of $\varphi$ at $x$ is defined $\mu$-almost everywhere, and the transport map $T$ is defined $\mu$-almost everywhere by
$$
T(x)=\operatorname{exp}_x(\frac{\widetilde\grad^g_x(\varphi)}{r^{1/(r-1)}\lVert \widetilde\grad^g_x(\varphi)\rVert_x^{(r-2)/(r-1)}}),
$$
where the approximate Riemannian gradient $\widetilde\grad^g(\varphi)$ of $\varphi$ is defined by
$$
g_x( \widetilde\grad^g_x(\varphi),\cdot)=\tilde d_x\varphi,
$$
and $\operatorname{exp}:TM\to M$ is the exponential map of $g$ on $TM$, which is globally defined since $M$ is complete.}
\end{theorem}
\begin{proof}
We first remark that
\begin{align*}
d^r(x,y)&\leq [d(x,x_0)+d(x_0,y)]^r\\
&\leq  [2\max (d(x,x_0),d(x_0,y))]^r\\
&\leq 2^r[d(x,x_0)^r+d(y,x_0)^r].
\end{align*}
Therefore
\begin{align*}
\int_{M\times M}d^r(x,y)\,d\mu(x)d\nu(y)&\leq \int_{M\times M}2^r[d(x,x_0)^r+d(y,x_0)^r]\,d\mu(x)d\nu(y)\\
&=2^r\int_Md^r(x,x_0)\,d\mu(x)+2^r\int_Md^r(y,x_0)\,d\nu(y)<\i,
\end{align*}
and thus the infimum in the Kantorovitch problem (\ref{kantprob}) with cost $d^r$ is finite.

By Example \ref{ExampleRiemannian}, the Lagrangian $L_{r,g}(x,v)=\lVert v\rVert_x^r=g_x(v,v)^{r/2}$ is a weak Tonelli Lagrangian. By Proposition
\ref{RiemannianUFC'}, the non-negative and continuous cost $d^r(x,y)$ is precisely the cost $c_{1,L_{r,g}}$.
Therefore this cost is locally semi-concave by Theorem \ref{costsemiconc}.
By Proposition \ref{uniformlocalsemiconcave}, this implies that $d^r(x,y)$ satisfies
condition {\rm (i)} of Theorem \ref{mainthm}. The fact that the cost $d^r(x,y)$ satisfies
the left twist condition {\rm (ii)} of Theorem \ref{mainthm} follows from Proposition
\ref{RiemannianUFC'}. Therefore there is an optimal transport map $T$.

If the measure $\mu$ is absolutely continuous with respect to Lebesgue measure, and
$(\varphi,\psi)$ is a calibrated subsolution for the cost $d^r(x,y)$ and the pair of measures $(\mu,\nu)$, then by Complement \ref{compmainthm}, for $\mu$-almost every $x$, we have $(x,T(x))\in{\cal D}(\Lambda ^l_{c_{1,L_{r,g}}})$, and
$$
\frac{\partial c_{t,L_{r,g}}}{\partial x}(x,T(x))=-\tilde d_x\varphi.
$$
Since $(x,T(x))$ is in ${\cal D}(\Lambda ^l_{c_{1,L_{r,g}}})$, it follows from Proposition
\ref{RiemannianUFC'} that $T(x)=\pi\phi^g_1(x,v_x)$, where $\pi:TM\to M$ is the canonical projection, the flow $\phi^g_t$ is the geodesic flow of $g$ on $TM$, and
$v_x\in T_xM$ is determined by
$$
\frac{\partial c_{t,L_{r,g}}}{\partial x}(x,T(x))=-\frac{\partial L_{r,g}}{\partial v}(x,v_x),
$$
or, given the equality above, by
$$
\frac{\partial L_{r,g}}{\partial v}(x,v_x)=\tilde d_x\varphi.
$$
Now the vertical derivative of $L_{r,g}$ is computed in Example \ref{ExampleRiemannian}
$$
\frac{\partial L_{r,g}}{\partial v}(x,v)=r\lVert v\rVert_x^{r-2}g_x( v,\cdot).
$$
Hence $v_x\in T_xM$ is determined by
$$
r\lVert v_x\rVert_x^{r-2}g_x( v_x,\cdot)=\tilde d_x\varphi=g_x(\widetilde\grad^g_x(\varphi),\cdot).
$$
This gives the equality
$$
r\lVert v_x\rVert_x^{r-2}v_x=\widetilde\grad^g_x(\varphi),
$$
from which we easily get
$$
v_x=\frac{\widetilde\grad^g_x(\varphi)}{r^{1/(r-1)}\lVert \widetilde\grad^g_x(\varphi)\rVert_x^{(r-2)/(r-1)}}.
$$
Therefore
$$
T(x)=\pi\phi^g_t(x,\frac{\widetilde\grad^g_x(\varphi)}{r^{1/(r-1)}\lVert \widetilde\grad^g_x(\varphi)\rVert_x^{(r-2)/(r-1)}}).
$$
By definition of the exponential map $\exp:TM\to M$, we have
$\operatorname{exp}_x(v)=\pi\phi^g_t(x,v)$, and the formula for $T(x)$ follows.
\end{proof}

\section{The interpolation and its absolute continuity}
\label{sectinterp}
For a cost $c_{t,L}$ coming from a Tonelli Lagrangian $L$, Theorem
\ref{thLagragiancost} shows not only that we have an optimal transport map $T$
 but also that this map is obtained by following an extremal for time $t$.
 We can therefore interpolate the optimal transport by maps $T_s$ where
 we stop at intermediary times $s\in [0,t]$. We will show in this section that
 these maps are also optimal transport maps for costs coming from the same
 Lagrangian. Let us give now precise definitions.

For the sequel of this section, we consider $L$ a Tonelli Lagrangian on the connected manifold $M$. We fix $t>0$ and
$\mu_0$ and $\mu_t$ two probability measures on $M$, with $\mu_0$ absolutely continuous with respect to Lebesgue measure, and such that
$$
\min_{\gamma \in \Pi(\mu_0,\mu_t)} \left\{ \int_{M \times M} c_{t,L}(x,y)\,d\gamma(x,y) \right\}<+\i.
$$
We call $T_t$ the optimal transport map given by Theorem \ref{thLagragiancost} for $(c_{t,L},\mu_0,\mu_t)$.
We denote by $(\varphi,\psi)$ a fixed $(c_{t,L},\gamma_t)$-calibrated pair. Therefore $\gamma_t=(\operatorname{Id}_M \tilde\times T_t)_\# \mu_0$
is the unique optimal plan from $\mu_0$ to $\mu_t$.
By Theorem \ref{thLagragiancost}, we can find a Borel subset $B\subset M$ such that:
\begin{itemize}
\item the subset $B$ is of full $\mu_0$ measure;
\item the approximate ddifferential $\tilde d_x\varphi$ exists for every $x\in B$, and is Borel measurable on $B$;
\item the map $T_t$ is defined at every $x\in B$, and we have
$$
T_t(x)=\pi\phi^L_t(x,\widetilde\grad^L_x(\varphi)),
$$
where $\phi_L^t$ is the Euler-Lagrange flow of $L$,  $\pi:TM \rightarrow M$ is the canonical projection, and the Lagrangian approximate gradient $x\mapsto \widetilde\grad^L_x(\varphi))$ is defined by
$$
\frac{\partial L}{\partial v}(x,\widetilde\grad^L_x(\varphi))=\tilde d_x\varphi;
$$
\item for every $x\in B$, the partial derivative $\displaystyle\frac{\partial c}{\partial x}(x,T_t(x))$ exists,
and is uniquely defined by
$$
\frac{\partial c}{\partial x}(x,T_t(x))=-\tilde d_x\varphi;
$$
\item for every $x\in B$ there exists a unique $L$-minimizer $\gamma_x:[0,t]\to M$ between $x$ and $T_t(x)$. This $L$-minimizer $\gamma_x$ is given by
$$
\forall s\in[0,t],\quad
\gamma_ x(s)=\pi\phi^L_s(x,\widetilde\grad^L_x(\varphi));
$$
\item for every $x\in B$, we have
$$
\psi(T_t(x))-\varphi(x)=c_{t,L}(x,T_t(x)).
$$
\end{itemize}
For $s\in[0,t]$, we can therefore define an interpolation $T_s:M\to M$ defined on $B$, and hence $\mu_0$ almost everywhere, by
$$
T_s(x)=\gamma_ x(s)=\pi\phi^L_s(x,\widetilde\grad^L_x(\varphi)).
$$
Each map $T_s$ is Borel measurable. In fact, since the global Legendre transform is a homeomorphism and the
approximate differential is Borel measurable,
the Lagrangian approximate gradient $\widetilde\grad^L(\varphi)$ is itself Borel measurable.
Moreover the map $\pi\phi_s^L:TM\to M$ is continuous, and thus $T_s$ is Borel measurable.
We can therefore define the probability measure $\mu_s=T_{s\#}\mu_0$ on $M$, i.e.\ the measure $\mu_s$ is the image of $\mu_0$ under the Borel measurable map $T_s$.
\begin{theorem}\label{interpolation}{\sl Under the hypotheses above, the maps $T_s$ satisfies the following properties:
\begin{itemize}
\item[{\rm (i)}] For every $s\in ]0,t[$, the map $T_s$ is the (unique) optimal transport maps for the cost $c_{s,L}$ and the pair of measures $(\mu_0,\mu_s)$.
\item[{\rm(ii)}] For every $s\in ]0,t[$,
the map $T_s: B\to M$ is injective. Moreover, if we define $\bar c_{s,L}(x,y) =c_{s,L}(y,x)$,
the inverse map  $T_s^{-1}:T_s( B)\to B$ is the (unique) optimal transport map for the cost $\bar c_{s,L}$
and the pair of measures $(\mu_s,\mu_0)$, and it is
countably Lipschitz (i.e. there exist a Borel countable partition of
$M$ such that $T_s^{-1}$ is Lipschitz on each set).
\item[{\rm (iii)}] For every $s\in ]0,t[$, the measure $\mu_s=T_{s\#}\mu_0$ is absolutely continuous with respect to Lebesgue measure.
\item [{\rm (iv)}] For every $s\in ]0,t[$, the composition $\hat T_s=T_tT_s^{-1}$ is the (unique) optimal transport map for the cost $c_{t-s,L}$ and the pair of measures $(\mu_s,\mu_t)$,
and it is countably Lipschitz.
\end{itemize}}
\end{theorem}
\begin{proof} Fix $s\in ]0,t[$. It is not difficult to see, from the definition of $c_{t,L}$,  that
\begin{equation*}
\forall x,y,z\in M, \quad
c_{t,L}(x,y)\leq c_{s,L}(x,y)+c_{t-s,L}(y,z),\tag{1}
\end{equation*}
and even that
\begin{equation*}
\forall x,z\in M, \quad
c_{t,l}(x,z)=\inf_{y\in M} c_{s,L}(x,y)+c_{t-s,L}(y,z).
\end{equation*}
If $\gamma:[a,b]\to M$ is an $L$-minimizer, the restriction
$\gamma\vert_{[c,d]}$ to a subinterval $[c,d]\subset [a,b]$ is also an  $L$-minimizer.
In particular, we obtain
$$
\forall s\in ]a,b[, \quad
c_{b-a,L}(\gamma(a),\gamma(b))=c_{s-a,L}(\gamma(a),\gamma(s))+c_{b-s,L}(\gamma(s),\gamma(b)).
$$
Applying this to the $L$-minimizer $\gamma_x$, we get
\begin{equation*}
\forall x\in B, \quad
c_{t,L}(x,T_t(x))=c_{s,L}(x,T_s(x))+c_{t-s,L}(T_s(x),T_{t}(x)).\tag{2}
\end{equation*}
We define for every $s\in ]0,t[ $, two probability measures $\gamma_s,\hat \gamma_s$ on $M\times M$, by
$$
\gamma_s =(\operatorname{Id}_M \tilde\times T_s)_\#\mu_0 \quad
\text{and} \quad \hat\gamma_s =(T_s \tilde\times T_t)_\#\mu_0,
$$
where $\operatorname{Id}_M \tilde\times T_s$ and $T_s \tilde\times T_t$ are the maps from the subset $B$ of full $\mu_0$ measure to
$M\times M$ defined by
\begin{align*}
&(\operatorname{Id}_M \tilde\times T_s)(x) =(x, T_s(x)),\\
&(T_s \tilde\times T_t)(x) =(T_s(x),T_t(x)).
\end{align*}
Note that the marginals of $\gamma_s$ are $(\mu_0,\mu_s)$, and  those of $\hat\gamma_s$ are $(\mu_s,\mu_t)$. We claim that $c_{s,L}(x,y)$ is integrable for $\gamma_s$ and $\hat\gamma_{t-s}$. In fact, we have $C=\inf_{TM} L>-\infty$, hence $c_{r,L}\geq Cr$.
Therefore, the equality  (2) gives
\begin{equation*}\forall x\in B, \quad
[c_{t,L}(x,T_t(x))-Ct]=[c_{s,L}(x,T_s(x))-Cs]+[c_{t-s,L}(T_s(x),T_{t}(x))-C(t-s)].
\end{equation*}
Since the functions between brackets are all non-negative, we can integrate this equality with respect to $\mu_0$ to obtain
$$\int_{M\times M}[c_{t,L}(x,y)-Ct]\,d\gamma_t=\int_{M\times M}[c_{s,L}(x,y)-Cs]\,d\gamma_s+\int_{M\times M}[c_{t-s,L}(x,y)-C(t-s)]\,d\hat\gamma_s.$$
But all numbers involved in the equality above are non-negative, all measures are probability measures, and the cost $c_{t,L}$ is $\gamma_t$ integrable since $\gamma_t$ is an optimal plan for $(c_{t,L},\mu_0,\mu_t)$,
and the optimal cost of $(c_{t,L},\mu_0,\mu_t)$ is finite. Therefore we obtain that
$c_{s,L}$ is $\gamma_s$-integrable, and $c_{t-s,L}$ is $\hat\gamma_s$-integrable.

Since by definition of a calibrating pair we have $\varphi>-\infty$ and $\psi<+\infty$ everywhere on $M$,
we can find an increasing sequence of compact subsets $K_n\subset M$ with $\cup_nK_n=M$, and
we consider $V_n =K_n \cap \{\varphi\geq -n\}$, $V_n' =K_n \cap \{\psi\leq n\}$,
so that $\cup_n V_n=\cup_n V_n'=M$.

We define the functions $\varphi^n_s,\psi^n_s:M\to\R$ by
\begin{align*}
\psi^n_s(z)&=\inf_{\tilde z\in V_n}\varphi(\tilde z)+c_{s,L}(\tilde z,z),\\
\varphi^n_s(z)&=\sup_{\tilde z\in V'_n}\psi(\tilde z)-c_{t-s,L}(z, \tilde z),
\end{align*}
where $(\varphi,\psi)$ is the fixed $c_{t,L}$-calibrated pair. Note that $\psi_s^n$ is bounded from below
by $-n+t\inf_{TM}L>-\infty$.  Moreover, the family of functions
$(\varphi(\tilde z)+c_{s,L}(\tilde z,\cdot))_{\tilde z\in V'_n}$ is locally uniformly semi-concave with a linear modulus,
since this is the case for the family of functions
$(c_{s,L}(\tilde z,\cdot))_{\tilde z\in V'_n}$, by Theorem \ref{costsemiconc} and Proposition
\ref{uniformlocalsemiconcave}. It follows from Proposition
\ref{InfSemiConcaveSurCompact} that $\psi_s^n$  is semi-concave with a linear modulus.
A similar argument proves that  $-\varphi^n_s$ is semi-concave with a linear modulus.
Note also that, since $V_n$ and $V'_n$ are both increasing sequences, we have
$\psi^n_s\geq \psi^{n+1}_s$ and $\varphi^{n+1}_s\geq \varphi^{n}_s$, for every $n$.
Therefore we can define $\varphi_s$ (resp.\ $\psi_s$)  as the pointwise limit of the sequence $\varphi_s^n$

Using the fact that $(\varphi,\psi)$ is  a $c_{t,L}$-subsolution, and inequality (1) above,
we obtain
$$
\forall x,y,z\in M, \quad
\psi(y)-c_{t-s,L}(z,y) \leq \varphi(x)+c_{s,L}(x,z).
$$
Therefore we obtain for $x\in V_n, y\in V'_n,z\in M$
\begin{equation*}
\psi(y)-c_{t-s,L}(z,y) \leq \varphi_s^n(z) \leq \varphi_s(z) \leq \psi_s(z)
\leq \psi_s^n(z) \leq \varphi(x)+c_{s,L}(x,z).\tag{3}
\end{equation*}
Inequality (3) above yields
\begin{equation*}
\forall x,y,z\in M, \quad
\psi(y)-c_{t-s,L}(z,y) \leq \varphi_s(z) \leq \psi_s(z) \leq \varphi(x)+c_{s,L}(x,z).\tag{4}
\end{equation*}
In particular, the pair $(\varphi,\psi_s)$ is a $c_{s,L}$-subsolution, and the pair $(\varphi_s,\psi)$ is a $c_{t-s,L}$-subsolution.
Moreover, $\varphi$, $\psi_s$, $\varphi_s$ and $\psi$ are all Borel measurable.

We now define
$$
B_n =B\cap V_n\cap T_t^{-1}(V'_n),
$$
so that $\cup_n B_n=B$ has full $\mu_0$-measure.

If $x\in B_n$, it satisfies $x\in V_n$ and $T_t(x)\in V'_n$. From Inequality (3) above, we obtain
\begin{multline*}
\psi(T_t(x))-c_{t-s,L}(T_s(x),T_t(x)) \leq \varphi_s^n(T_s(x)) \leq \varphi_s(T_s(x)) \\
\leq \psi_s(T_s(x)) \leq \psi_s^n(T_s(x)) \leq \varphi(x)+c_{s,L}(x,T_s(x))
\end{multline*}
Since $B_n\subset B$, for $x\in B_n$, we have $\psi(T_t(x))-\varphi(x)=c_{t,L}(x,T_t(x))$.
Combining this with Equality (2), we conclude that the two extreme terms in the inequality above are equal. Hence, for every $x\in B_n$, we have
\begin{multline*}
\psi(T_t(x))-c_{t-s,L}(T_s(x),T_t(x)) = \varphi_s^n(T_s(x)) = \varphi_s(T_s(x)) \\
= \psi_s(T_s(x)) = \psi_s^n(T_s(x)) = \varphi(x)+c_{s,L}(x,T_s(x)).\tag{5}
\end{multline*}
In particular, we get
$$
\forall x\in B, \quad
\psi_s(T_s(x))=\varphi(x)+c_{s,L}(x,T_s(x)),
$$
or equivalently
$$
\psi_s(y)-\varphi(x) = c_{s,L}(x,y) \quad \text{for  $\gamma_s$-a.e. $( x,y)$}.
$$
Since $(\varphi,\psi_s)$ is a (Borel) $c_{s,L}$-subsolution,
it follows that the pair $(\varphi,\psi_s)$ is  $(c_{s,L},\gamma_s)$-calibrated.
Therefore, by Theorem \ref{thmduality} we get that
$\gamma_s=(\operatorname{Id}_M \tilde\times T_s)_\#\mu_0$ is an optimal plan for
$(c_{s,L},\mu_0,\mu_s)$. Moreover, since $c_{s,L}$ is $\gamma_s$-integrable, the
infimum in the Kantorovitch problem (\ref{kantprob}) in Theorem \ref{mainthm} with cost $c_{s,L}$ is finite, and therefore there exists
a unique optimal transport plan. This proves (i).

Note for further reference that a similar argument, using the equality
$$
\forall x\in B, \quad
\psi(T_t(x))=\varphi_s(T_s(x))+c_{t-s,L}(T_s(x),T_t(x)),
$$
which follows from Equation (5) above, shows that the measure $\hat\gamma_s=(T_s \tilde\times T_t)_\#\mu_0$
is an optimal plan for the cost $c_{t-s,L}$ and the pair of measures $(\mu_s,\mu_t)$.

We now want to prove (ii). Since $B$ is the increasing union of $B_n=B\cap V_n\cap T_t^{-1}(V'_n)$,
it suffices to prove that $T_s$ is injective on $B_n$ and that the restriction
$T^{-1}|_{T(B_n)}$ 
is locally Lipschitz on $T_s(B_n)$.

Since $B_n\subset V_n$, by Inequality (3) above
we have
\begin{equation*}
\forall x\in B_n,\forall y\in M, \quad
\varphi^n_s(y)\leq \psi^n_s(y)\leq \varphi(x)+c_{s,L}(x,y).\tag{6}
\end{equation*}
Moreover, by Equality (5) above
\begin{equation*}
\forall x\in B_n,\quad
\varphi^n_s(T_s(x))=\psi^n_s(T_s(x))= \varphi(x)+c_{s,L}(x,T_s(x)).\tag{7}
\end{equation*}
In particular, we have $\varphi^n_s\leq \psi^n_s$ everywhere with equality at every point of
$T_s(B_n)$. As we have said above, both functions $\psi^n_s$ and $-\varphi^n_s$ are locally semi-concave with a
linear modulus. It follows, from Theorem \ref{CriterionForDiff}, that both derivatives
$d_z\varphi^n_s,d_z\psi^n_s$ exist and are equal for $z\in T_s(B_n)$.
Moreover, the map
$z\mapsto d_z\varphi^n_s=d_z\psi^n_s$ is locally Lipschitz on $T_s(B_n)$.
Note that we also get from (6) and (7) above that for a fixed $x\in B_n$, we have
$\varphi^n_s\leq \varphi(x)+c_{s,L}(x, \cdot)$ everywhere with equality at $T_s(x)$.
Since $\varphi_n$ is semi-convex, using that $c_{s,L}(x,\cdot)$ is semi-concave,
again by Theorem \ref{CriterionForDiff}, we obtain that the partial derivative
$\displaystyle\frac{ \partial c_{s,L}}{\partial y}(x,T_s(x))$ of $c_{s,L}$
with respect to the second variable exists and is equal to $d_{T_s(x)}\varphi^n_s=d_{T_s(x)}\psi^n_s$.
Since $\gamma_x:[0,t]\to M$ is an $L$-minimizer with $\gamma_x(0)=x$ and
$\gamma_x(s)=T_s(x)$, it follows from Corollary \ref{ComputationSuperdifferential} that
$$
d_{T_s(x)}\psi^n_s=\frac{ \p c_{s,L}}{\partial y}(x,T_s(x))=\frac{\partial L}{\partial v}(\gamma_x(s),\dot\gamma_x(s)).
$$
Since $\gamma_x$ is an $L$-minimizer, its speed curve is an orbit of the Euler-Lagrange flow, and therefore
$$
(T_s(x),d_{T_s(x)}\psi^n_s)=\Leg ((\gamma_x(s),\dot\gamma_x(s))=\Leg \phi^L_s(\gamma_x(0),\dot \gamma_x(0)),
$$
and
$$
x=\pi\phi^L_{-s}\Leg^{-1}(T_s(x),d_{T_s(x)}\psi^n_s).
$$
It follows that $T_s$ is injective on $B_n$ with inverse given by the map $\theta_n:T_s(B_n)\to B_n$ defined, for $z\in T_s(B_n)$, by
$$
\theta_n(z) =\pi\phi^L_{-s}\Leg^{-1}(z,d_z\psi^n_s).
$$
Note that the map $\theta_n$ is locally Lipschitz on $T_s(B_n)$, since this is the case for
$z\mapsto d_z\psi^n_s$, and both maps $\phi^L_{-s},\Leg^{-1}$ are C$^1$, since $L$ is a Tonelli Lagrangian.
An analogous argument proves the countably Lipshitz regularity of $\hat T_s=T_tT_s^{-1}$ in part (iv).
Finally the optimality of $T_s^{-1}$ simply follows from
\begin{align*}
\min_{\gamma \in \Pi(\mu_s,\mu_0)} \left\{ \int_{M \times M} \bar c_{s,L}(x,y)\,d\gamma(x,y) \right\}
&=\min_{\gamma \in \Pi(\mu_0,\mu_s)} \left\{ \int_{M \times M} c_{s,L}(x,y)\,d\gamma(x,y) \right\}\\
&=\int_M c_{s,L}(x,T_s(x))\,d\mu_0(x)\\
&=\int_M \bar c_{s,L}(y,T_s^{-1}(y))\,d\mu_s(y).
\end{align*}

Part (iii) of the Theorem follows from part (ii). In fact, if $A\subset M$ is Lebesgue negligible, the image $T_s^{-1}(T_s(B)\cap A)$ is also Lebesgue negligible, since
$T_s^{-1}$ is countably Lipschitz on $T_s(B)$, and therefore
sends Lebesgue negligible subsets to  Lebesgue negligible subsets.
It remains to note, using that $B$ is of full $\mu_0$-measure, that
$\mu_s(A)=T_{s\#}\mu_0(A)=\mu_0(T_s^{-1}(T_s(B)\cap A))=0$.

To prove part (iv), we already know that $\hat \gamma_s=(T_s \tilde\times T_t)_\#\mu_0$ is an optimal plan for the cost $c_{t-s,L}$ and the measures $(\mu_s,\mu_t)$. Since the Borel set
$B$ is of full $\mu_0$-measure, and $T_s:B\to T_s(B)$ is a bijective Borel measurable map,
we obtain that $T^{-1}_{s}$ is a Borel map, and
$\mu_0=T^{-1}_{s\#}\mu_s$. It follows that
$$
\hat \gamma_s=(\operatorname{Id}_M \tilde\times T_tT^{-1}_s)_\#\mu_s.
$$
Therefore the composition $T_tT^{-1}_s$ is an optimal transport map for the cost $c_{t-s,L}$ and the pair of measures $(\mu_s,\mu_t)$,
and it is the unique one since $c_{t-s,L}$ is $\hat\gamma_s$-integrable
and $\mu_s$ is absolutely continuous with respect to the Lebesgue measure.
\end{proof}

\begin{remark}
\label{rmkintegrability}{\rm
We observe that, in proving the uniqueness statement in parts (i) and (iv) of the above theorem,
we needed the full generality of Theorem \ref{thLagragiancost}.
Indeed, assuming
$$
\int_{M\times M}c_{t,L}(x,y)\,d\mu_0(x)d\mu_t(y)<+\infty,
$$
there is a priori no reason for which the two integrals
$$
\int_{M\times M}c_{s,L}(x,y)\,d\mu_0(x)d\mu_s(y), \quad
\int_{M\times M}c_{t-s,L}(x,y)\,d\mu_s(x)d\mu_t(y)
$$
would have to be finite.
So the existence and uniqueness of a transport map in Theorem \ref{mainthm}
under the integrability assumption on $c$ with respect to $\mu \otimes \nu$
instead of assumption (iv) would not have been enough to obtain Theorem
\ref{interpolation}.}
\end{remark}

\begin{remark}{\rm
We remark that,
if both $\mu_0$ and $\mu_t$ are not assumed to be absolutely continuous, and therefore no optimal transport map
necessarily exists, one can still define
an ``optimal'' interpolation $(\mu_s)_{0 \leq s \leq t}$ between $\mu_0$ and $\mu_t$
using some measurable selection theorem (see \cite[Chapter 7]{villani2}). Then, adapting our proof,
one still obtains that, for any $s \in (0,t)$, there exists a unique optimal transport map
$S_s$ for $(\bar c_{s,L},\mu_s,\mu_0)$ (resp. a unique optimal transport map
$\hat S_s$ for $(c_{t-s,L},\mu_s,\mu_t)$), and this map is countably Lipschitz.

We also observe that, if the manifold is compact, our proof shows that the above maps are globally Lipschitz (see \cite{bernbuf}).}
\end{remark}

\bigskip
\bigskip
\bigskip
\appendix
\centerline{\Large \bf Appendix}
\section{Semi-concave functions}
We give the definition of semi-concave function and we recall their principal properties.
The main reference on semi-concave functions is the book \cite{cansin}.

We first recall the definition of a modulus (of continuity).

\begin{definition}[Modulus]{\rm A \textit{modulus} $\omega$ is a continuous non-decreasing function
$\omega: [0,+\infty) \rightarrow [0,+\infty)$ such that $\omega(0)=0$.

We will say that a modulus is {\it linear} if it is of the form $\omega(t)=kt$, where $k\geq 0$ is some fixed constant.}
\end{definition}
We will need the notion of superdifferential. We define it in an intrinsic way on a manifold.
\begin{definition}[Superdifferential]
\label{defsuperdiff}{\rm Let $f:M \to \R$ be a function.
We say that $p \in T^*_xM$ is a \textit{superdifferential} of $f$ at $x\in M$, and we write $p \in D^+f(x)$,
if there exists a function $g:V \rightarrow \R$, defined on some open subset $U\subset M$ containing $x$, such that $g \geq f$, $g(x)=f(x)$, and
$g$ is differentiable at $x$ with $d_xg=p$.}
\end{definition}
We now give the definition of a semi-concave function on an open subset of a Euclidean space.
\begin{definition}[Semi-concavity]\label{defsemicon}{\rm
Let $U \subset \R^n$ open.
A function $f:U \rightarrow \R$ is said to be \textit{semi-concave} in $U$ with modulus $\omega$
(equivalently $\omega$-\textit{semi-concave}) if,
for each $x \in U$, we have
$$
f(y) - f(x) \leq \< l_x,y-x\> + \lVert y-x\rVert \omega(\lVert y-x\rVert)
$$
for a certain linear form $l_x: \R^n \rightarrow \R$.\\
Note that necessarily $l_x\in D^+f(x)$.
Moreover we say that $f:U \rightarrow \R$ is \textit{locally semi-concave} if, for each $x \in U$,
there exists an open neighborhood of
$x$ in which $f$ is semi-concave for a certain modulus.

We will say that the  function $f:U\to \R$ is locally semi-concave with a linear modulus if,
for each $x\in U$, we can find an open neighborhood  $V_x$ such that the restriction $f|_{V_x}$ is $\omega$-semi-concave, with $\omega$ a linear modulus.}
\end{definition}
\begin{proposition}\label{EasyPropSemiConcave}{\sl 1) Suppose $f_i:U\to \R, i=1,\dots,k$ is $\omega_i$-semi-concave, where $U$ is an open subset of $\R^n$. Then we have:
\begin{itemize}
\item[\rm (i)] for any $\alpha_1,\dots,\alpha_k\geq 0$, the functions $\sum_{i=1}^k\alpha_if_i$ is $(\sum_{i=1}^k\alpha_i\omega_i)$-semi-concave on $U$.
\item[\rm (ii)] the function $\min_{i=1}^kf_i$ is $(\max_{i=1}^k\omega_i)$-semi-concave.
\end{itemize}
2) Any {\rm C}$^1$ function is locally semi-concave.}
\end{proposition}
\begin{proof} The proof of 1)(i) is obvious. For the proof of (ii), we fix $x\in U$, and we find $i_0\in \{1,\dots,k\}$ such that $\min_{i=1}^kf_i(x)=f_{i_0}(x)$. Since $f_{i_0}$ is $\omega_{i_0}$-semi-concave, we can find a linear map $l_x:\R^n\to \R$ such that
$$
\forall y\in U, \quad f_{i_0}(y)-f_{i_0}(x)\leq l_x(y-x)+\lVert y-x\rVert\omega_{i_0}(\lVert y-x\rVert).
$$
It clearly follows that
$$
\forall y\in U, \quad \min_{i=1}^kf_{i}(y)-\min_{i=1}^kf_{i}(x)\leq l_x(y-x)+\lVert y-x\rVert\max_{i=1}^k\omega_{i}(\lVert y-x\rVert).
$$
To prove 2), consider an open convex subset $C$ with $\bar C$ compact and contained in $U$. By compactness of $\bar C$  and continuity of $x\mapsto d_xf$, we can find a modulus $\omega$, which is a modulus of continuity for the map $x\mapsto d_xf$ on $C$. The Mean Value Formula in integral form
$$
f(y)-f(x)=\int_0^1d_{tx+(1-t)y}f(y-x)\,dt,
$$
which is valid for every $y,x\in C$ implies that
$$
\forall x,y\in U, \quad f(y)-f(x)\leq d_xf(y-x)+ \lVert y-x\rVert\omega(\lVert y-x\rVert).
$$
Therefore $f$ is $\omega$-semi-concave in the open subset $C$.
\end{proof}

We now state and prove the first important consequences of the definition of semi-concavity.

\begin{lemma}
\label{semiconcaveimpliesLipschitz}{\sl Suppose $U$ is an open subset of $\R^n$.
Let $f:U \rightarrow \R$ be an $\omega$-semi-concave function. Then we have:
\begin{enumerate}
\item[(i)] for every compact subset $K\subset U$, we can find a constant $A$ such that
for every $x\in K$, and every linear form $l_x$ on $\R^n$ satisfying
$$\forall y\in U, \quad f(y)-f(x)\leq \< l_x,y-x\>+\lVert y-x\rVert\omega(\lVert y-x\rVert),$$
we have $\lVert l_x \rVert \leq A$;
\item[(ii)] the function
$f$ is locally Lipschitz.
\end{enumerate}}
\end{lemma}
\begin{proof} From the definition, it follows that a semi-concave function is locally bounded from above.
We now show that $f$ is also locally bounded from below.
Fix a (compact) cube $C$ contained in $U$ and let $\{ y_1, \ldots, y_{2^n}\}$ be the vertices of the cube.
Then, for each $x \in C$, we can write $x= \sum_i \a_i y_i$, with $\sum_i \a_i=1$.
By the semi-concavity of $f$ we have, for each $i=1,\ldots,2^n$,
$$
f(y_i) - f(x) \leq \< l_x,y_i-x\> + \lVert y_i-x\rVert \omega(\lVert y_i-x\rVert);
$$
multiplying by $\a_i$ and summing over $i$, we get
$$
\sum_i \a_i f(y_i) \leq f(x) + \sum_i \a_i \lVert y_i-x \rVert \omega(\lVert y_i-x \rVert) \leq f(x) + B,
$$
with $B=D_C\omega(D_C)$, where $D_C$ is the diameter of the compact cube $C$.
It follows that
$$
\forall x\in C, \quad f(x) \geq \min_i f(y_i) - B.
$$
We now know that $f$ is locally bounded. Using this fact, it is not
difficult to show (i). In fact, suppose that the closed ball $\bar B(x_0,2r), r<+\infty$, is contained in $U$.
For $x \in \bar B(x_0,r)$, we have $x - r v \in \bar B(x_0,2r)\subset U$
for each $v \in \R^n$ with $\lVert v \rVert=1$, and therefore
$$
f(x-rv) - f(x) \leq \< l_x,-rv\> +   \lVert -rv \rVert \omega(\lVert -rv \rVert) = -r\< l_x,v\> +r\omega(r) .$$
Since, by the compactness of $\bar B(x_0,2r)$, we already know that $\tilde B=\sup_{z\in \bar B(x_0,2r)} \lvert f(z)\rvert$ is finite,
this implies
$$
\< l_x,v\> \leq \frac{f(x) - f(x-rv)}{r} + \omega(r)
\leq \frac{2\tilde B}{r} + \omega(r).
$$
It follows that, for $x\in  \bar B(x_0,r)$,
$$\lVert l_x \rVert \leq \frac{2\tilde B}{r} + \omega(r).$$
Since the compact set $K\subset U$ can be covered by a finite numbers of balls $ \bar B(x_i,r_i)$, $i=1,\dots,\ell$, we obtain (i).

To prove (ii), we consider a compact subset $K\subset U$, and we apply (i) to obtain the constant $A$. We denote by $D_K$ the (finite) diameter of the compact set $K$. For each $x,y \in K$,
\begin{align*}
f(y) - f(x) &\leq \< l_x,y-x\> + \lVert y-x\rVert \omega(\lVert y-x\rVert)\\
& \leq
\left(\lVert l_x\rVert + \omega(D_K)\right)\lVert y-x\rVert\\
& \leq (A + \omega(D_K))\lVert y-x\rVert.
\end{align*}
Exchanging the role of $x$ and $y$, we conclude  that $f$ is Lipschitz on $K$.
\end{proof}
Let us recall that a Lipschitz real valued function defined on an open subset of a Euclidean space is differentiable almost everywhere
(with respect to the Lebesgue measure). Therefore by part (ii) of Lemma
\ref{semiconcaveimpliesLipschitz} above we obtain the following corollary:
\begin{corollary}{\sl A locally semi-concave real valued function defined on an open subset of a Euclidean space is differentiable almost everywhere with respect to the Lebesgue measure.}
\end{corollary}
In fact, in the case of semi-concave functions there is a better result which
is given in  Theorem \ref{propdiffsemiconc} below,
whose proof can be found in \cite[Section 4.1]{cansin}.
Let us first give a definition:

\begin{definition}
\label{defcountLip}{\rm
We say that $E \subset \R^n$ is \textit{countably $(n-1)$-Lipschitz} if there exists a countable family of compact subsets $K_j$ is a compact subset of $\R^{n}$ such that:
\begin{enumerate}
\item $E$ is contained in $\cup_j K_j$;
\item
for each $j$ there exists a hyperplane $H_j \subset \R^n=H_j\oplus H_j^\perp$, where  $H_j^\perp$ is the Euclidean orthogonal of $H_j$, such that $K_j$ is contained in the graph of a Lipschitz function $f_j:A_j \to H_j^\perp$
defined on a compact subset $A_j\subset H_j$.
\end{enumerate}}
\end{definition}
Note that in the definition above, by the graph property (ii), the compact subset $K_j$ has finite $(n-1)$-dimensional Hausdorff measure.
Therefore any $(n-1)$-Lipschitz set is contained in a Borel (in fact $\s$-compact) $(n-1)$-Lipschitz set
with $\s$-finite $(n-1)$-dimensional Hausdorff measure.
\begin{theorem}
\label{propdiffsemiconc}{\sl If $\varphi:U\to \R$ is a semi-concave function defined on the open subset $U$ of $\R^n$, then
$\varphi$ is differentiable at each point in the complement of a Borel countably $(n-1)$-Lipschitz set.}
\end{theorem}

In order to extend the definition of locally semi-concave to functions defined on a manifold,
it suffices to show that this definition is stable by composition with diffeomorphisms.

\begin{lemma}{\sl
Let $U,V \subset \R^n$ be open subsets. Suppose that $F:V \to U$ is  a {\rm C}$^1$ map.
If $f:U \to \R$ is a locally semi-concave function then $f\circ F: V\to \R$ is also locally semi-concave.
Moreover, if $F$ is of class {\rm C}$^2$, and
$f:U \to \R$ is a locally semi-concave function with a linear modulus then $f\circ F: V\to \R$ is also locally semi-concave with a linear modulus.}
\end{lemma}
\begin{proof} Since the nature of the result is local, without loss of generality we can assume that
$f:U \to \R$ is semi-concave with modulus $\o$. We now show that, for every $V'$ convex open subset
whose closure $\bar V'$ is compact and contained in $V$,
the restriction $f \circ F|_{V'}:V' \rightarrow \R$ is a semi-concave function. We set
$C_{\bar V'}= \max_{z\in \bar V'}\lVert D_zF\rVert$,  and we denote by $\hat\omega_{\bar V'}$ a modulus
of continuity for the continuous function
$z\mapsto D_zF$ on the compact subset $\bar V'$.

For each $x,y $ in the compact convex subset $\bar V'\subset V$, we have
\begin{align*}
f(F(y)) - f(F(x)) & \leq \< l_{F(x)},F(y)-F(x)\> + \lVert F(y)-F(x)\rVert \omega(\lVert F(y)-F(x)\rVert)\\
&\leq \< l_{F(x)},D F(x)(y-x)\> + \lVert l_{F(x)}\rVert \hat\omega_{\bar V'}(\lVert y-x\rVert) \lVert y-x\rVert\\
&+ C_{\bar V'}\lVert y-x\rVert \omega(C_{\bar V'}\lVert y-x\rVert);
\end{align*}
Since $F(\bar V')$ is a compact subset of $U$ we can apply part (i) of Lemma \ref{semiconcaveimpliesLipschitz} to obtain that
$\tilde C_{\bar V'}=\sup_{\bar V'}\lVert l_{F(x)}\rVert$ is finite. This implies that $f \circ F$ on $V'$ is semi-concave
with the modulus
$$\tilde \o(r)=\tilde C_{\bar V'}\hat\o_{\bar V'}(r) + C_{\bar V'} \omega(C_{\bar V'}r).$$
If $F$ is {\rm C}$^2$, then its derivative $DF$ is locally Lipschitz on $U$, and we  can assume that $\hat\omega_{\bar V'}$ is a linear modulus. Therefore, if $\omega$ is a linear modulus, we obtain that $\tilde \omega$ is also a linear modulus.
\end{proof}

Thanks to the previous lemma, we can define a locally semi-concave function (resp.\ a locally semi-concave function for a linear modulus) on a manifold as a function whose restrictions to charts is, when computed in coordinates, locally semi-concave (resp.\ locally semi-concave for a linear modulus). Moreover, it suffices to check this locally semi-concavity in charts for a family of charts whose domains of definition cover the manifold.
It is not difficult to see that Theorem \ref{propdiffsemiconc} is valid on any (second countable) manifold, since we can cover such a manifold by the domains of definition of a countable family of charts.

Now we want to introduce the notion of uniformly semi-concave family of functions.
\begin{definition}{\rm
Let $f_i:U \rightarrow \R$, $i \in I$, be a family of functions defined on an open subset $U$ of $\R^n$.
We will say that the family $(f_i)_{i \in I}$ is \textit{uniformly $\omega$-semi-concave}, where $\omega$ is a modulus of continuity,
if each $f_i$ is $\omega$-semi-concave. We will say that the family $(f_i)_{i \in I}$ is \textit{uniformly semi-concave} if there exists a
modulus of continuity $\omega$ such that the family $(f_i)_{i \in I}$ is uniformly $\omega$-semi-concave.
We will say that the family $(f_i)_{i \in I}$ is \textit{uniformly semi-concave with a linear modulus}, if it is uniformly $\omega$-semi-concave, with $\omega$ of the form $t\mapsto kt$, where $k$ is a fixed constant. }
\end{definition}

\begin{theorem}\label{InfSemiConcave}{\sl
Suppose that $f_i:U \rightarrow \R$, $i \in I$, is a family of functions defined on an open subset $U$ of $\R^n$.
Suppose that this family $(f_i)_{i \in I}$ is uniformly $\omega$-semi-concave, where $\omega$ is a modulus of continuity.
If the function
$$
f(x) = \inf_{i \in I} f_i(x)
$$
is finite everywhere on $U$, then $f:U \rightarrow \R$ is also $\omega$-semi-concave.}
\end{theorem}
\begin{proof}
Fix $x_0 \in U$. We can find a sequence $i_n$ such that $f_{i_n}(x_0)\searrow f(x_0)>-\infty$.
We choose a cube $C \subset U$ with center $x_0$. Call $y_1, \ldots, y_{2^n}$ the vertices of $C$.
By the argument in the beginning of the proof of Lemma \ref{semiconcaveimpliesLipschitz}, we have
$$
\forall x \in C, \ \forall i \in I, \quad
\min_{1 \leq j \leq 2^n} f_i(y_j) \leq f_i(x_0) + D_C \omega(D_C),
$$
where $D_C$ is the diameter of the compact cube $C$. Using the fact that $f(y_j)=\inf_{i \in I}f_i(y_j)$ is finite,
it follows that there exists $A \in \R$ such that
$$
\forall x \in C, \ \forall i \in I, \quad f_i(x) \geq A.
$$
Choose now $\e > 0$ such that $\bar B(x_0,\e) \subset C$. If $l_i : \R^n \rightarrow \R$ is a linear form such that
$$
\forall y \in U, \quad f_i(y) \leq  f_i(x_0) + \< l_i,y-x_0\> + \lVert y-x_0\rVert \omega(\lVert y-x_0\rVert),
$$
we obtain that, for every $v \in \R^n$ of norm $1$,
$$
A \leq f_i(x_0) + \< l_i,\e v\> + \e \omega(\e).
$$
Since $f_{i_n}(x_0)\searrow f(x_0)$, we can assume $f_{i_n}(x_0) \leq M<+\i$ for all $n$, that implies
$$
\lVert l_{i_n} \rVert \leq \frac{M-A}{\e} + \o(\e) <+\infty.
$$
Up to extracting a subsequence, we can assume $l_{i_n} \rightarrow l$ in $\R^{n*}$,
the dual space of $\R^n$.
Then, as for every $y \in U$ we have $f(y) \leq f_{i_n}(y)$, passing to the limit in $n$ in the inequality
$$
f(y) \leq f_{i_n}(x_0) + \< l_{i_n},y-x_0\> + \lVert y-x_0\rVert \omega(\lVert y-x_0\rVert),
$$
we get
$$
f(y) \leq f(x_0) + \< l,y-x_0\> + \lVert y-x_0\rVert \omega(\lVert y-x_0\rVert).
$$
Since $x_0 \in U$ is arbitrary, this concludes the proof.
\end{proof}

Before generalizing the notion of uniformly semi-concave family of functions to manifolds,
let us look at the following example.

\begin{example}{\rm
For $k \in \R$, define $f_k:\R \rightarrow \R$ as $f_k(x)=kx$.
It is clear that the family $(f_k)_{k \in \R}$ is $\omega$-semi-concave for every modulus of continuity $\omega$.
In fact
$$
f_k(y)-f_k(x)=k(y-x) \leq k(y-x) + |y-x| \o(|y-x|),
$$
since $\omega \geq 0$.
Consider now the diffeomorphism $\varphi : \R^*_+ \rightarrow \R^*_+$, $\varphi(x)=x^2$.
Then there does not exist a non-empty open subset $U \subset \R^*_+$, and a modulus of continuity $\omega$,
such that the family $(f_k \circ \varphi|_{U})_{k \in \R}$ is (uniformly) $\omega$-semi-concave.
Suppose in fact, by absurd, that
$$
f_k\circ \varphi(y)-f_k\circ \varphi(x) \leq l_x(y-x) + |y-x| \o(|y-x|),
$$
where $l_x$ depends on $k$ but not $\omega$.
Since $f_k\circ \varphi$ is differentiable we must have $l_x(y-x)=(f_k\circ \varphi)'(x)(y-x)=2kx(y-x)$.
Therefore we should have
$$
ky^2 - kx^2 \leq 2kx(y-x) + |y-x| \o(|y-x|).
$$
Fix $x,y \in U$, with $y \neq x$ and set $h=y-x$. Then
$$
kh^2\leq |h| \o(|h|) \Rightarrow k \leq \frac{\o(|h|)}{|h|} \ \ \ \forall k,
$$
that is obviously absurd.}
\end{example}

Therefore the following is the only reasonable definition for the notion of a uniformly locally semi-concave
family of functions on a manifold.

\begin{definition}{\rm
We will say that the family of functions $f_i:M \rightarrow \R$, $i \in I$, defined on the manifold $M$, is
\textit{uniformly locally semi-concave} (resp.\ \textit{with a linear modulus}), if we can find a cover $(U_j)_{j \in J}$ of $M$ by open subsets, with
each $U_j$ domain of a chart $\varphi_j:U_j \stackrel{\sim}{\longrightarrow} V_j \subset \R^n$ (where $n$ is the
dimension of $M$), such that for every $j \in J$ the family of functions $(f_i \circ \varphi_j^{-1})_{i \in I}$
is a uniformly semi-concave family of functions on the open subset $V_j$ of $\R^n$
(resp. with a linear modulus).}
\end{definition}

The following corollary is an obvious consequence of Theorem \ref{InfSemiConcave}.\begin{corollary}\label{corstability}{\sl
If the family $f_i:M \rightarrow \R$, $i \in I$ is uniformly locally semi-concave (resp.\ with a linear modulus) and the function
$$
f(x) = \inf_{i \in I} f_i(x)
$$
is finite everywhere, then $f:M \rightarrow \R$ is locally semi-concave (resp.\ with a linear modulus).}
\end{corollary}

\begin{definition}{\rm
Suppose $c:M \times N \rightarrow \R$ is a function defined on the product of the manifold $M$ by the topological space $N$.
We will say that the family of functions $(c(\cdot,y))_{y \in N}$ is \textit{locally uniformly locally semi-concave}
(resp.\ \textit{with a linear modulus)}, if for each
$y_0 \in N$ we can find a neighborhood $V_0$ of $y_0$ in $N$ such that the family $(c(\cdot,y))_{y \in V_0}$
is uniformly locally semi-concave on $M$ (resp.\ with a linear modulus).}
\end{definition}
\begin{proposition}\label{InfSemiConcaveSurCompact}{\sl Suppose $c:M \times N \rightarrow \R$ is a function defined on the product of the manifold $M$ by the  topological space $N$, such that the family of functions $(c(\cdot,y))_{y \in N}$ is locally uniformly locally semi-concave (resp.\ with a linear modulus). If $K\subset N$ is compact,
and the function
$$
f_K(x) =\inf_{y\in K}c(x,y)
$$
is finite everywhere on $U$, then $f_K:U\to \R$
is locally semi-concave on M (resp.\ with a linear modulus).}
\end{proposition}
\begin{proof} By compactness of $K$, we can find a finite family $V_i,i=1,\dots,\ell$ of open subsets of $N$ such that $K\subset \cup_{i=1}^\ell V_i$, and for every $i=1,\dots,\ell$, the family $(c(\cdot,y))_{y \in V_i}$ is locally uniformly locally semi-concave (resp.\ with a linear modulus).
The function
$$
f_{i}(x) =\inf_{y\in K\cap V_i}c(x,y)
$$
is finite everywhere on $U$, because $f_{i}\geq f_{K}$. It follows from Corollary
\ref{corstability} that $f_{i}$ is locally semi-concave on M (resp.\ with a linear
 modulus), for $i=1,\dots,\ell$. Since $f_{K}=\min_{i=1}^\ell f_{i}$, we can apply part (ii)
 of Proposition \ref{EasyPropSemiConcave} to conclude that $f_K$ has the same property.
\end{proof}
\begin{proposition}\label{uniformlocalsemiconcave}{\sl If $c:M \times N \rightarrow \R$ is a locally semi-concave function (resp.\ with a linear modulus) on the product of the manifolds $M$ and $N$, then
the family of functions on $M$ $(c(\cdot,y))_{y \in N}$ is locally uniformly locally semi-concave (resp.\ with a linear modulus).}
\end{proposition}
\begin{proof}
We can cover $M \times N$ by a family $(U_i \times W_j)_{i \in I,j \in J}$ of open sets with $U_i$ open in $M$, $W_j$ open in $N$,
where $U_i$ is the domain of a chart $\varphi_i:U_i \stackrel{\sim}{\longrightarrow} \tilde U_i \subset \R^n$ (where $n$ is the
dimension of $M$), and $W_j$ is the domain of a chart $\psi_j:W_j \stackrel{\sim}{\longrightarrow} \tilde W_j \subset \R^m$ (where $m$ is the
dimension of $M$), and such that
$$
(\tilde x,\tilde y) \mapsto c\left(\varphi_i^{-1}(\tilde x),\psi_j^{-1}(\tilde y)\right)
$$
is $\omega_{i,j}$-semi-concave on $\tilde U_i \times \tilde W_j$, for some modulus $\omega_{i,j}$. It is then clear that the family
$$
(c(\varphi_i^{-1}(\tilde x),\psi_j^{-1}(\tilde y)))_{\tilde y \in \tilde W_j}
$$
is uniformly locally $\omega_{i,j}$-semi-concave on $\tilde U_i$.
\end{proof}
The following corollary is now an obvious consequence of Propositions
 \ref{uniformlocalsemiconcave} and \ref{InfSemiConcaveSurCompact}.
\begin{corollary}\label{coruniformlocalsemiconcave}{\sl Suppose $c:M \times N \rightarrow \R$ is a locally semi-concave function (resp.\ with a linear modulus) on the product of the manifolds $M$ and $N$. Let $K$ be a compact subset of $N$. If the function
$$
f_K(x) =\inf_{y\in K} c(x,y)
$$
is finite everywhere on $U$, then $f_K:U\to\R$ is locally uniformly locally semi-concave (resp.\ with a linear modulus).}
\end{corollary}
We end this section with another useful theorem. The proof we give is an adaptation of the proof of \cite[Lemma 3.8, page 494]{fathiToulouse}.
\begin{theorem}\label{CriterionForDiff}{\sl Let $\varphi_1,\varphi_2:M\to \R$ be two functions, with $\varphi_1$ locally semi-convex (i.e. $-\varphi_1$ locally semi-concave), and $\varphi_2$ locally semi-concave.
Assume that $\varphi_1\leq\varphi_2$. If we define ${\cal E} =\{x\in M\mid  \varphi_1(x)=\varphi_2(x)\}$,
then both $\varphi_1$ and $\varphi_2$ are differentiable at each $x\in {\cal E}$ with $d_x\varphi_1=d_x\varphi_2$
at such a point. Moreover, the map $x\mapsto d_x\varphi_1=d_x\varphi_2$ is continuous on $\cal E$.

If  $\varphi_1$ is  locally semi-convex and $\varphi_2$ is locally semi-concave,  both with a linear modulus, then, in fact, the map $x\mapsto d_x\varphi_1=d_x\varphi_2$ is locally Lipschitz on $\cal E$.}
\end{theorem}
\begin{proof} Since the statement is local in nature, we will assume that $M= \stackrel{\circ}{\B}$ is the Euclidean unit ball of center $0$ in $\R^n$, and that $-\varphi_1$ and $\varphi_2$ are semi-concave with (common) modulus $\omega$. Suppose now that $x\in{\cal E}$. We can find
two linear maps $l_{1,x},l_{2,x}:\R^n\to\R$ such that
\begin{align*}
\varphi_1(y)&\geq \varphi_1(x)+l_{1,x}(y-x)-\lVert y-x\rVert_{\rm euc}\omega(\lVert y-x\rVert_{\rm euc})\\
\varphi_2(y)&\leq \varphi_2(x)+l_{2,x}(y-x)+\lVert y-x\rVert_{\rm euc}\omega(\lVert y-x\rVert_{\rm euc}).
\end{align*}
Using $\varphi_1\leq \varphi_2$, and $\varphi_1(x)=\varphi_2(x)$, we obtain
\begin{multline*}
l_{1,x}(y-x)-\lVert y-x\rVert_{\rm euc}\omega(\lVert y-x\rVert_{\rm euc})\leq \varphi_1(y)- \varphi_1(x)\leq\\
\leq \varphi_2(y)-\varphi_2(x)\leq  l_{2,x}(y-x)+\lVert y-x\rVert_{\rm euc}\omega(\lVert y-x\rVert_{\rm euc}).\tag{1}
\end{multline*}
In particular, we get
$$ l_{1,x}(y-x)-\lVert y-x\rVert_{\rm euc}\omega(\lVert y-x\rVert_{\rm euc})\leq l_{2,x}(y-x)+\lVert y-x\rVert_{\rm euc}\omega(\lVert y-x\rVert_{\rm euc}),$$
replacing $y$ by $x+v$ with $\lVert v\rVert_{\rm euc}$  small, we conclude
$$l_{1,x}(v)-\lVert v\rVert_{\rm euc}\omega(\lVert v\rVert_{\rm euc})\leq l_{2,x}(v)+\lVert v\rVert_{\rm euc}\omega(\lVert v\rVert_{\rm euc}).$$
Therefore
$$
\lvert[ l_{2,x}-l_{1,x}](v)\rvert\leq 2\lVert v\rVert_{\rm euc}\omega(\lVert v\rVert_{\rm euc}),
$$
for $v$ small enough. Since $l_{2,x}-l_{1,x}$ is linear it must be identically $0$.
We set $l_x=l_{2,x}=l_{1,x}$. For $i=1,2$ and $y\in \stackrel{\circ}{\B}$, we obtain from (1)
\begin{equation*}
\lvert  \varphi_i(y)- \varphi_i(x)-l_x(y-x)\rvert\leq
\lVert y-x\rVert_{\rm euc}\omega(\lVert y-x\rVert_{\rm euc}).\tag{2}
\end{equation*}
This implies that $\varphi_i$ is differentiable at $x\in {\cal E}$, with $d_x\varphi_i=l$.
It remains to show the continuity of the derivative.
Fix $r<1$. We now find a modulus of continuity of the derivative on the ball $r\stackrel{\circ}{\B}$. If $y_1,y_2\in{\cal E}\cap r\stackrel{\circ}{\B}$, and $\lVert k\rVert_{\rm euc}\leq 1-r$, we can apply three times (2) to obtain
\begin{align*}
\varphi_1(y_2)- \varphi_1(y_1)-d_{y_1}\varphi_1(y_2-y_1) & \leq \lVert y_2-y_1\rVert_{\rm euc}\omega(\lVert y_2-y_1\rVert_{\rm euc})\\
\varphi_1(y_2+k)- \varphi_1(y_2)-d_{y_2}\varphi_1(k) & \leq \lVert k\rVert_{\rm euc}\omega(\lVert k\rVert_{\rm euc})\\
-\varphi_1(y_2+k)+ \varphi_1(y_1)+d_{y_1}\varphi_1(y_2+k-y_1) & \leq \lVert y_2+k-y_1\rVert_{\rm euc}\omega(\lVert y_2+k-y_1\rVert_{\rm euc}).
\end{align*}
If we add the first two inequality to the third one, we obtain
\begin{align*}
[d_{y_1}\varphi_1-d_{y_2}\varphi_1](k) &\leq \lVert y_2-y_1\rVert_{\rm euc}\omega(\lVert y_2-y_1\rVert_{\rm euc})
+\lVert k\rVert_{\rm euc}\omega(\lVert k\rVert_{\rm euc})\\
&+[\lVert y_2-y_1\rVert_{\rm euc}+\lVert k\rVert_{\rm euc}]\omega(\lVert y_2-y_1\rVert_{\rm euc}+\lVert k\rVert_{\rm euc}),
\end{align*}
which implies, exchanging $k$ with $-k$, and using that the modulus $\omega$
is non-decreasing
\begin{equation*}
\lvert [d_{y_1}\varphi_1-d_{y_2}\varphi_1](k)\rvert\leq
2[\lVert y_2-y_1\rVert_{\rm euc}+\lVert k\rVert_{\rm euc}]\omega(\lVert y_2-y_1\rVert_{\rm euc}+\lVert k\rVert_{\rm euc}).
\end{equation*}
Since $\lVert y_2-y_1\rVert_{\rm euc}<2$, we can apply the inequality (3) above with any $k$ such that
$\lVert k\rVert_{\rm euc}=(1-r)\lVert y_2-y_1\rVert_{\rm euc}/2$. If we divide the inequality
(3) by $\lVert k\rVert_{\rm euc}$, and take the sup over all $k$ such that
$\lVert k\rVert_{\rm euc}=(1-r)\lVert y_2-y_1\rVert_{\rm euc}/2$, we obtain
$$
\lVert d_{y_1}\varphi_1-d_{y_2}\varphi_1\rVert_{\rm euc}\leq
2\Bigl[\frac{2}{1-r}+1 \Bigr]\omega(\bigl(1+\frac{1-r}2 \bigr)\lVert y_2-y_1\rVert_{\rm euc}).
$$
It follows that a modulus of continuity of $x\mapsto d_x\varphi_1$ on ${\cal E}\cap r\stackrel{\circ}{\B}$ is given by
$$
t\mapsto \frac{6-2r}{1-r}\omega(\frac{3-r}2t).
$$
This implies the continuity of the map $x\mapsto d_x\varphi_1$ on ${\cal E}\cap r\stackrel{\circ}{\B}$. It also shows that it is Lipschitz on ${\cal E}\cap r\stackrel{\circ}{\B}$ when $\omega$ is a linear modulus.
\end{proof}

\section{Tonelli Lagrangians}
\subsection{Definition and background}
We recall some of the basic definition, and some of the results in Calculus of variations (in one variable). There are a lot of references on the subject. In \cite{fathi}, one can find an introduction to the subject that is particularly suited for our purpose. Other references are
\cite{butgiahil} and the first chapters in \cite{Moserbook}. A brief and
particularly nice  description of the main results is contained in \cite{ClarkeSIAM}.
\begin{definition}[Lagrangian]{\rm If $M$ is a manifold, a \textit{Lagrangian} on $M$ is a function
$L:TM\to \R$. In the following we will assume that $L$ is at least bounded below and continuous.}
\end{definition}
\begin{definition}[Action]{\rm If $L$ is a Lagrangian on $M$, for an absolutely continuous curve
$\gamma:[a,b]\to M, a\leq b$, we can define its \textit{action} ${\mathbb A}_L(\gamma)$ by
$$
{\mathbb A}_L(\gamma) =\int_a^bL(\gamma(s),\dot\gamma(s))\, ds.
$$}
\end{definition}
Note that the integral is well defined with values in $\R\cup\{+\infty\}$, because $L$ is bounded below, and $s\to L(\gamma(s),\dot\gamma(s))$ is defined a.e. and measurable.
To make things simpler to write, we set ${\mathbb A}_L(\gamma)=+\infty$ if $\gamma$ is not absolutely continuous.
\begin{definition}[Minimizer]\label{Minimizer}{\rm If $L$ is a Lagrangian on the manifold $M$,
an absolutely continuous curve $\gamma:[a,b]\to M$, with $a\leq b$, is an $L$-\textit{minimizer},
if $\mathbb A_L(\gamma)\leq \mathbb A_L(\delta)$ for every absolutely continuous curve $\delta:[a,b]\to M$
with the same endpoints, i.e. such that $\delta(a)=\gamma(a)$ and $\delta(b)=\gamma(b)$.}
\end{definition}
\begin{definition}[Tonelli Lagrangian]\label{TonelliLagrangian}{\rm We will say that $L:TM \rightarrow \R$ is a \textit{weak Tonelli Lagrangian} on $M$, if it satisfies the following hypotheses:
\begin{enumerate}
\item[(a)] $L$ is C$^1$;
\item[(b)] for each $x \in M$, the map $L(x,\cdot):T_xM \rightarrow \R$ is strictly convex;
\item[(c)] there exist a complete Riemannian metric $g$ on $M$ and a constant $C>-\infty$ such that
$$
\forall (x,v) \in TM, \quad L(x,v) \geq \lVert v \rVert_x + C
$$
where $\lVert \cdot \rVert_x$ is the norm on $T_x M$ obtained from the Riemannian metric $g$;
\item[(d)] for every compact subset $K \subset M$ the restriction of $L$ to $T_KM  = \cup_{x \in K} T_xM$
is superlinear in the fibers of $TM \rightarrow M$:
this means that for every $A\geq 0$, there exists a constant $C(A,K)>-\infty$ such that
$$
\forall (x,v) \in T_KM, \quad L(x,v) \geq A \lVert v \rVert_x + C(A,K).
$$
\end{enumerate}
We will say that $L$ is a \textit{Tonelli Lagrangian}, if it is a weak Tonelli Lagrangian, and satisfies the following two strengthening of conditions (a) and (b) above:
\begin{enumerate}
\item[(a')] $L$  is {\rm C}$^2$;
\item[(b')] for every
$(x,v) \in TM$, the second partial derivative $\displaystyle\frac{\partial^2 L}{\partial v^2}(x,v)$ is positive definite on $T_xM$.
\end{enumerate}}
\end{definition}
Since above a compact subset of a manifold all Riemannian metrics are equivalent, if condition (d) in the definition is satisfied for one particular Riemannian metric, then it is satisfied for any other Riemannian metric.

Note that when $L$ is a weak Tonelli Lagrangian on $M$, and $U:M\to \R$ is a C$^1$ function {\em which is bounded below}, then $L+U$, defined by $(L+U)(x,v)=L(x,v)+U(x)$ is a weak Tonelli Lagrangian. If moreover $L$ is a Tonelli Lagrangian, and $U$ is C$^2$ and bounded below, then $L+U$ is a Tonelli Lagrangian. Therefore one can generate a lot of (weak) Tonelli Lagrangians from the following simple example.
\begin{example}\label{ExampleRiemannian}{\rm Suppose that $g$ is a complete smooth Riemannian metric on $M$, and $r>1$. We define the Lagrangian $L_{r,g}$ on $M$ by
$$
L_{r,g}(x,v)=\lVert v\rVert_x^r=g_x(v,v)^{r/2}.
$$
\begin{enumerate}
\item[1)] $L_{2,g}$ is a Tonelli Lagrangian.
\item[2)] For any $r>1$, the Lagrangian is C$^1$ and is a weak Tonelli Lagrangian.
\end{enumerate}
In both cases, the Riemannian metric mentioned in condition (c) of Definition
\ref{TonelliLagrangian} is the same metric $g$.

Moreover, the vertical derivative of the Lagrangian $L_{r,g}$ is given by
$$
\frac{\partial L_{r,g}}{\partial v}(x,v)=r\lVert v\rVert_x^{r-2}g_x( v,\cdot),
$$
}
\end{example}
\begin{proof} Since $r>1$ it is not difficult
to check that $L$ has (in coordinates) partial derivatives everywhere with
$$
\frac{\partial L_{r,g}}{\partial x}(x,0)=0\quad
\text{and} \quad \frac{\partial L_{r,g}}{\partial v}(x,0)=0,
$$
and that these partial derivatives are continuous. Therefore $L$ is C$^1$. A simple computation gives
$$
\frac{\partial L_{r,g}}{\partial v}(x,v)=r\lVert v\rVert_x^{r-2}g_x( v,\cdot).
$$
We now prove condition (c) and (d) of Definition \ref{TonelliLagrangian} at once. In fact, if $A$ is given, we have
$$
L_{r,g}(x,v)=\lVert v\rVert_x^r\geq A\lVert v\rVert_x-A^{r/r-1},
$$
as on can see by considering separately the two cases $\lVert v\rVert_x^{r-1}\geq A$ and
$\lVert v\rVert_x^{r-1}\leq A$. The rest of the proof is easy.
\end{proof}
The completeness of the Riemannian metric in condition (c) of Definition
\ref{TonelliLagrangian} above is crucial to guarantee that a set of the form
$$
{\cal F} =\{\gamma\in C^0([a,b],M) \mid \gamma(a)\in K,\ {\mathbb A}_L(\gamma)\leq \kappa\},
$$
where $K$ is a compact subset in $M$, $\kappa$ is a finite constant, and  $a\leq b$, is compact in the C$^0$ topology. In fact, condition (c) implies that the curves in such a set
$\cal F$ have a $g$-length which is bounded independently of $\gamma$. Since $K$ is compact (assuming $M$ connected to simplify things) this implies that there exist $x_0\in M$ and $R<+\infty$ such that all the curves in $\cal F$ are contained in the closed ball $\bar B(x_0,R)=\{y\in M\mid d(x,y)\leq R\}$, where $d$ is the distance associated to the Riemannian metric $g$. But such a ball $\bar B(x_0,R)$ is compact since $g$ is complete (Hopf-Rinow Theorem). From there, one obtains that the set ${\cal F}$ is compact in the C$^0$ topology, see  \cite[Chapters 2 and 3]{butgiahil}.

The direct method in the Calculus of Variations, see \cite[Theorem 3.7, page 114]{butgiahil} or \cite{fathi} for Tonelli Lagrangians, implies:
\begin{theorem}\label{existenceminimizer}{\sl Suppose $L$ is a weak Tonelli Lagrangian on the connected manifold $M$. Then for every
$\ a,b \in \R$, $a<b$, and every $x,y \in M$, there exists an absolutely continuous curve $\gamma:[a,b] \rightarrow M$ which is an  $L$-minimizer with $\gamma(a)=x$ and $\gamma(b)=y$.}
\end{theorem}
In fact in \cite[Theorem 3.7, page 114]{butgiahil}, the existence of absolutely continuous minimizers is valid under very general hypotheses on the Lagrangian $L$ (the C$^1$ hypothesis on $L$ is much stronger than necessary).
We now come to the problem of regularity of minimizers which uses the {\rm C}$^1$ hypothesis on $L$:
\begin{theorem}\label{regularityminimizer}{\sl If $L$ is a weak Tonelli Lagrangian, then every minimizer $\gamma:[a,b]\to M$ is {\rm C}$^1$. Moreover, on every interval $[t_0,t_1]$ contained in a domain of a chart, it satisfies the following equality written in the coordinate system
\begin{equation*}
\frac{\partial L}{\partial v}(\gamma(t_1),\dot\gamma(t_1))
- \frac{\partial L}{\partial v}(\gamma(t_0),\dot\gamma(t_0))
= \int_{t_0}^{t_1} \frac{\partial L}{\partial x}(\gamma(s),\dot\gamma(s))\,ds,\tag{IEL}
\end{equation*}
which is an integrated from of the Euler-lagrange equation. This implies that ${\partial L}/{\partial v}(\gamma(t),\dot\gamma(t))$ is a {\rm C}$^1$ function of $t$ with
$$\frac{d}{dt}\left[\frac{\partial L}{\partial v}(\gamma(t),\dot\gamma(t))\right]=\frac{\partial L}{\partial x}(\gamma(t),\dot\gamma(t)).$$

Moreover, if $L$ is a {\rm C}$^r$ Tonelli Lagrangian, with $ r\geq 2$, then any minimizer is of class C$^r$.}
\end{theorem}
\begin{proof} We will only sketch the proof. If $L$ is a Tonelli Lagrangian, this theorem would be a formulation of what is nowadays called Tonelli's existence and regularity theory. In that case its proof can be found in many places, for example \cite{butgiahil}, \cite{ClarkeSIAM}, or \cite{fathi}.
The fact that the regularity of minimizers holds for C$^1$ (or
even less smooth) Lagrangians is more recent. The fact that a
minimizer is Lipschitz has been established by Clarke and Vinter,
see \cite[Corollary 1, page 77, and Corollary 3.1, page
90]{ClarkeVinter} (again the hypothesis $L$ is C$^1$ is stronger
than the one required in this last work). The same fact under
weaker regularity assumptions on $L$ has been proved in
\cite{amasbut}. A short and elegant proof of the fact that a
minimizer for the class of absolutely continuous curves is
necessarily Lipschitz has been given by Clarke, see
\cite{ClarkeETDS}.
The key idea of why
$\dot\gamma$ should be bounded is that the Energy is conserved
along absolutely continuous minimizers (this is easy to prove for C$^1$ minimizers, see Corollary
\ref{corConsEnergy}) and the sublevels of the Energy are compact
(see Proposition \ref{propEnergyCpt}).

Once one knows that $\gamma$ is Lipschitz, when
$L$ is C$^1$ it is possible to differentiate the action, see
\cite{butgiahil}, \cite{ClarkeSIAM}, or \cite{fathi}, and, using
an integration by parts, one can show that $\gamma$ satisfies the
following integrated form (IEL') of the Euler-Lagrange equation
for  almost every $t\in[t_0,t_1]$, for some fixed linear form $c$:
\begin{equation*}
\frac{\partial L}{\partial v}(\gamma(t),\dot\gamma(t))
=c+ \int_{t_0}^{t} \frac{\partial L}{\partial x}(\gamma(s),\dot\gamma(s))\,ds.\tag{IEL'}
\end{equation*}
But the continuity of the right hand side in (IEL') implies that ${\partial L}/{\partial v}(\gamma(t),\dot\gamma(t))$ extends continuously everywhere on $[t_0,t_1]$. Conditions (a) and (b) on $L$ imply that the global Legendre transform
$$
\Leg: TM \rightarrow T^*M,
$$
$$
(x,v) \mapsto (x,\frac{\partial L}{\partial v}(x,v)),
$$
is continuous and injective, therefore a homeomorphism on its image by, for example,  Brouwer's Theorem on the Invariance of Domain (see also Proposition
\ref{LegendreHomeo} below). We therefore conclude that $\dot\gamma(t)$ has a continuous extension to $[t_0,t_1]$. Since $\gamma$ is Lipschitz this implies that $\gamma$ is C$^1$. Equation (IEL) follows from (IEL'), which now holds everywhere by continuity.
\end{proof}
In fact, in this paper, we will only use the cases when $L$ is C$^2$, in which case this regularity of minimizers will follow from the ``usual'' Tonelli regularity theory, or when $L$ is of the form $L(x,v)=\lVert v\rVert^p_x$, $p>1$, where the norm is obtained from a
C$^2$ Riemannian metric, in which case the minimizers are necessarily geodesics which are of course as smooth as the Riemannian metric, see Proposition
\ref{RiemannianUFC'} below.

To obtain further properties it is necessary to introduce the global Legendre transform.
\begin{definition}[Global Legendre Transform]
\label{defGlobLegTransf}
{\rm If $L$ is a C$^1$ Lagrangian on the manifold $L$, its \textit{global Legendre transform}
$\Leg: TM \rightarrow T^*M$, where $T^*M$ is the cotangent bundle of $M$,
is defined by
$$
\Leg (x,v)  =(x,\frac{\partial L}{\partial v}(x,v)).
$$}
\end{definition}
\begin{proposition}\label{LegendreHomeo}{\sl If $L$ is a weak Tonelli Lagrangian on the manifold $M$, then its global Legendre transform
$\Leg: TM \rightarrow T^*M$  is a homeomorphism from $TM$ onto $T^*M$.

Moreover, if $L$ is a {\rm C}$^r$ Tonelli Lagrangian with $r\geq 2$, then $\Leg$ is {\rm C}$^{r-1}$.}
\end{proposition}
\begin{proof} We first prove the surjectivity of $\Leg$. Suppose $p\in T^*_xM$. By condition (d) in Definition \ref{TonelliLagrangian}, we have
\begin{align*}
p(v)-L(x,v)&\leq p(v)-(\lVert p\rVert_x+1)\rVert v\lVert_x-C(\lVert p\rVert_x+1,\{x\})\\
&\leq -\rVert v\lVert_x-C(\lVert p\rVert_x+1,\{x\}).
\end{align*}
But this last quantity tends to $-\infty$, as ${\rVert v\lVert_x\to+\infty}$. Therefore the continuous function $v\mapsto p(v)-L(x,v)$ achieves a maximum at some
point $v_p\in T_xM$. Since this function is C$^1$, its derivative at $v_p$ must be $0$. This yields $p-\partial L/\partial v(x,v_p)=0$. Hence $ (x,p)=\Leg(x,v_p)$.

To prove injectivity of $\Leg$, it suffices to show that for $v,v'\in T_xM$, with $v\neq v'$, we have $\partial L/\partial v(x,v)\neq \partial L/\partial v(x,v')$. Consider the function
$\varphi:[0,1]\to \R,t\mapsto L(x,tv+(1-t)v')$, which by condition (b) of Definition \ref{TonelliLagrangian} is strictly convex. Since it is C$^1$, we must have $\varphi'(0)\neq \varphi'(1)$. In fact, if that was not the case, then the non-decreasing function $\varphi'$ would be constant on $[0,1]$, and $\varphi$ would be affine on $[0,1]$. This contradicts strict convexity. By a simple computation, we therefore get
$$
\frac{\partial L}{\partial v}(x,v')(v-v')=\varphi'(0)\neq\varphi'(1)=\frac{\partial L}{\partial v}(x,v)(v-v').
$$
This implies $\partial L/\partial v(x,v')\neq \partial L/\partial v(x,v)$.
We now show that $\Leg$ is a homeomorphism. Since this map is continuous, and bijective, we have to check that it is proper, i.e.\ inverse images under $\Leg$ of compact subsets of $T^*M$ are (relatively) compact. For this it suffices to show that for every compact subset $K \subset M$, and every $C <+\infty$, the set
$$
\{(x,v) \in TM \mid x \in K,\  \lVert \frac{\partial L}{\partial v}(x,v)\rVert_x\leq C\}
$$
is compact. By convexity of $v\mapsto L(x,v)$, we obtain
$$
\frac{\partial L}{\partial v}(x,v)(v)\geq L(x,v)-L(x,0).
$$
But $\lVert{\partial L}/{\partial v}(x,v)\rVert_x \geq {\partial L}/{\partial v}(x,v)(v/\lVert v\rVert_x)$, therefore by condition (d) of Definition \ref{TonelliLagrangian}, we conclude that
$$
\forall A\geq 0,\forall (x,v)\in T_KM, \quad
\lVert\frac{\partial L}{\partial v}(x,v)\rVert_x \geq A-[C(K,A)/\rVert v\lVert_x].
$$
Taking $A=C+1$,
we get the inclusion
$$
\{(x,v) \in TM \mid x \in K, \ \lVert \frac{\partial L}{\partial v}(x,v)\rVert_x\leq C\}
\subset \{(x,v) \in TM \mid x \in K, \ \lVert v \rVert_x\leq C(K,C+1)\},
$$
and the compactness of the first set follows.

Suppose now that $L$ is a {\rm C}$^r$ Tonelli Lagrangian with $r\geq 2$. Obviously $\Leg$ is {\rm C}$^{r-1}$. By the inverse function theorem, to show that it is a {\rm C}$^{r-1}$ diffeomorphism, it suffices to show that the derivative is invertible at each point of $TM$. But a simple computation in coordinates show that the derivative of $\Leg$ at $(x,v)$ is given in matrix form by
$$
\begin{pmatrix}
     {\rm Id} & 0\\
     \displaystyle\frac{ \partial ^2 L}{\partial x\partial v}(x,v) &
     \displaystyle\frac{ \partial ^2 L}{\partial v^2}(x,v) \\
   \end{pmatrix}
$$
This is clearly invertible by (b') of Definition \ref{TonelliLagrangian}.
\end{proof}
\begin{definition}{\rm If $L$ is a Lagrangian on $M$, we define its \textit{Hamiltonian} $H:T^*M\to\R\cup\{+\infty\}$ by
$$
H(x,p) =\sup_{v\in T_xM} p(v)-L(x,v).
$$}
\end{definition}
\begin{proposition}\label{propertiesH}{\sl Let $L$ be  a weak Tonelli Lagrangian on the manifold $M$. Its Hamiltonian $H$ is everywhere finite valued and satisfies the following properties:
\begin{enumerate}
\item[{\rm (a$^*$)}] $H$ is C$^1$, and in coordinates
\begin{equation*}
\left\{
\begin{array}{l}
\dfrac{\partial H}{\partial p}(\Leg (x,v))=v\\
\dfrac{\partial H}{\partial x}(\Leg (x,v))=-\dfrac{\partial L}{\partial x}(x,v).
\end{array}
\right.
\end{equation*}
\item[{\rm (b$^*$)}] for each $x \in M$, the map $H(x,\cdot):T^*_xM \rightarrow \R$ is strictly convex;
\item[{\rm (d$^*$)}] for every compact subset $K \subset M$ the restriction of $H$ to $T_K^*M  = \cup_{x \in K} T_x^*M$
is superlinear in the fibers of $T^*M \rightarrow M$:
this means that for every $A\geq 0$, there exists a finite constant $C^*(A,K)$ such that
$$
\forall (x,p) \in T_K^*M, \quad H(x,p) \geq A \lVert p \rVert_x + C^*(A,K).
$$
In particular, the function $H$ is a proper map, i.e.\ inverse images under H of compact subsets of $\R$ are compact.
\end{enumerate}
If $L$ is a {\rm C}$^r$ Tonelli Lagrangian with $r\geq 2$, then
\begin{enumerate}
\item[{\rm (a'$^*$)}] $H$  is {\rm C}$^r$;
\item[{\rm (b'$^*$)}] for every
$(x,v) \in M$, the second partial derivative $\displaystyle\frac{\partial^2 H}{\partial p^2}(x,p)$ is positive definite on $T_x^*M$.
\end{enumerate}}
\end{proposition}
\begin{proof} To show differentiability, using a chart in $M$, we can assume that $M=U$ is an open subset in $\R^m$. Moreover, since all Riemannian metrics are equivalent above compact subsets, replacing $U$ by an open subset $V$ with compact closure contained in $U$, we can assume that the norm used in (c) of Definition
 \ref{TonelliLagrangian} is the constant standard Euclidean norm
$\lVert\cdot\rVert_{\rm{euc}}$ on the second factor of $TV=V\times \R^m$, that is
$$
\forall x\in V,\ \forall v\in \R^m, \quad L(x,v)\geq A\lVert v\rVert_{\rm{euc}}+C(A),
$$
where $C(A)$ is a finite constant, and $\sup _{x\in V}L(x,0)\leq C<+\infty$.

We have $T^*V=V\times \R^{m*}$, where $\R^{m*}$ is the dual space of $\R^{m}$. We will denote by $\lVert\cdot\rVert_{\rm{euc}}$ also the dual norm on $\R^{m*}$ obtained from $\lVert\cdot\rVert_{\rm{euc}}$ on $\R^{m}$. We now fix $R>0$. If $p\in \R^{m*}$ satisfies $\lVert p\rVert_{\rm{euc}}\leq R$, we have
\begin{align*}
p(v)-L(x,v)&\leq \lVert p\rVert_{\rm{euc}}\lVert v\rVert_{\rm{euc}}-(R+1)\lVert v\rVert_{\rm{euc}}-C(R+1)\\
&\leq -\lVert v\rVert_{\rm{euc}}-C(R+1).
\end{align*}
Since $L(x,0)\leq C$ for $x\in V$, it follows that,
for $\lVert v\rVert_{\rm{euc}}>C-C(R+1)$,
$$
p(v)-L(x,v) \leq -C \leq -L(x,0).
$$
This implies
$$
H(x,p)=\sup_{v\in \R^m}p(v)-L(x,v)=\sup_{\lVert v\rVert_{\rm{euc}}\leq C-C(R+1)}p(v)-L(x,v),
$$
Therefore the sup in the definition of $H(x,p)$ is attained at a point $v_{(x,p)}$ with
$\lVert v_{(x,p)}\rVert_{\rm{euc}}\leq C-C(R+1)$. Note that this point $v_{(x,p)}$ is unique (compare with the argument proving that the Legendre transform is surjective). In fact, at its maximum $v_{(x,p)}$, the C$^1$ function $v\mapsto p(v)-L(x,v)$ must have $0$ derivative, and therefore
$$
p=\frac{\partial L}{\partial v}(x,v_{(x,p)}).
$$
This means $(x,p)=\Leg(x,v_{(x,p)})$, but the Legendre transform is injective by Proposition \ref{LegendreHomeo}.

Note, furthermore, that the map
$$
f:\bigl(V\times\{\lVert p\rVert_{\rm{euc}}\leq R\} \bigr)\times
\{\lVert v\rVert_{\rm{euc}}\leq C-C(R+1)\}\to \R,
$$
$$
((x,p),v)\mapsto p(v)-L(x,v),
$$
is {\rm C}$^1$.
Therefore we obtain that $H$ is C$^1$ from the following classical lemma whose proof is left to the reader.
\begin{lemma}{\sl Let $f:N\times K\to \R, (x,k)\mapsto f(x,k)$ be a continuous map, where $N$ is a manifold, and $K$ is a compact space.
Define $F:N\to\R$ by $F(x)=\sup_{k\in K}f(x,k)$.
Suppose that:
\begin{enumerate}
\item
$\displaystyle\frac{\partial f}{\partial x}(x,k)$ exists everywhere and is continuous as a function of both variables $(x,k)$;
\item for every $x\in N$, the set $\{k\in K\mid f(x,k)=F(x)\}$ is reduced to a single point, which we will denote by $k_x$.
\end{enumerate}
Then $F$ is {\rm C}$^1$, and the derivative $D_xF$ of $F$ at $x$ is given by
$$D_xF=\frac{\partial f}{\partial x}(x,k_x).$$}
\end{lemma}
Returning to the proof of Proposition \ref{propertiesH}, by the last statement of the above lemma
we also obtain
$$
\frac{\partial H}{\partial p}(x,p)=v_{(x,p)}\quad \text{and} \quad
\frac{\partial H}{\partial x}(x,p)=-\frac{\partial L}{\partial x}(x,v_{(x,p)})$$
Since $(x,p)=\Leg(x, v_{(x,p)})$, this can be rewritten as
\begin{equation}\label{dHdp}\frac{\partial H}{\partial p}\circ \Leg(x,v)=v\quad \text{and}
\quad \frac{\partial H}{\partial x}\circ \Leg(x,v)=-\frac{\partial L}{\partial x}(x,v),
\end{equation}
which proves (a$^*$).
Note that when $L$ is a C$^r$ Tonelli Lagrangian, by Proposition \ref{LegendreHomeo} the Legendre transform
$\Leg$ is a C$^{r-1}$ global diffeomorphism. From the expression of the partial derivatives above, we conclude that $\partial H/\partial p$ and $\partial H/\partial x$ are both
C$^{r-1}$. This proves (a'$^*$).

We now prove (b'$^*$). Taking the derivative in $v$ of the first equality in (\ref{dHdp})
$$
\frac{\partial H}{\partial p}\left[ x,\frac{\partial L}{\partial v}(x,v)\right]=v,
$$
we obtain the matrix equation
$$
\frac{\partial^2 H}{\partial p^2}(\Leg(x,v))\cdot\frac{\partial^2 L}{\partial v^2}(x,v)
={\rm Id}_{\R^m},
$$
where the dot $\cdot$ represents the usual product of matrices. This means that the matrix representative of ${\partial^2H}/{\partial p^2}(x,p)$ is the inverse of the matrix of a positive definite quadratic form, therefore ${\partial^2H}/{\partial p^2}(x,p)$ is itself positive definite.

We prove (b$^*$). Suppose $p_1\neq p_2$ are both in $T^*_xM$. Fix $t\in]0,1[$, and set $p_3=tp_1+(1-t)p_2$. The covectors $p_1,p_2,p_3$ are all distinct. Call $v_1,v_2,v_3$ elements in $T_xM$ such that $p_i=\partial L/\partial v(x,v_i)$. By injectivity of the Legendre transform, the tangent vectors $v_1,v_2,v_3$ are also all distinct. Moreover, for $i=1,2$ we have
$$
H(x,p_i)=p_i(v_i)-L(x,v_i),
$$
$$
H(x,p_3)=p_3(v_3)-L(x,v_3)= t[p_1(v_3)-L(x,v_3)]+(1-t)[p_2(v_3)-L(x,v_3)].
$$
Since the sup in the definition of $H(x,p)$ is attained at a unique point, and $v_1,v_2,v_3$ are all distinct, for $i=1,2$ we must have
$$
p_i(v_3)-L(x,v_3)<p_i(v_i)-L(x,v_i)=H(x,p_i).
$$
It follows that
$$
H(x,tp_1+(1-t)p_2)< tH(x,p_1)+ (1-t)H(x,p_2).
$$

It remains to prove (d$^*$). Fix a compact set $K$ in $M$. Since
$$
H(x,p)\geq p(v)-L(x,v),
$$
we obtain
$$
H(x,p)\geq \sup_{\rVert v\lVert_x\leq A}p(v)+
\inf_{x\in K,\rVert v\lVert_x\leq A}-L(x,v).
$$
But $C^*(A,K)=\inf_{x\in K,\rVert v\lVert_x\leq A}-L(x,v)$  is finite by compactness,
and $\sup_{\rVert v\lVert_x\leq A}p(v)=A\rVert p\lVert_x$.
\end{proof}
Since for a weak Tonelli Lagrangian $L$, the Hamiltonian $ H:T^*M\to\R$ is C$^1$, we can define the Hamiltonian vector field $X_H$ on $T^*M$. This is rather standard and uses the fact that the exterior derivative of the Liouville form on $M$ defines a symplectic form on $M$, see \cite{AbrahamMarsden} or \cite{HoferZehnder}. The vector field $X_H$ is entirely characterized by the fact that in coordinates obtained from a chart in $M$, it is given by
$$
X_H(x,p)=(\frac{\partial H}{\partial p}(x,p),-\frac{\partial H}{\partial x}(x,p)).
$$
So the associated ODE is given by
\begin{equation*}
\left\{
\begin{array}{l}
\dot x=\dfrac{\partial H}{\partial p}(x,p)\\
\dot p=-\dfrac{\partial H}{\partial x}(x,p).
\end{array}
\right.
\end{equation*}
In this form, it is an easy exercise to check that $H$ is constant on any solution of $X_H$.

We know come to the simple and important connection between minimizers and solutions of $X_H$.
\begin{theorem}\label{ConnectionWithH}{\sl Suppose $L$ is a weak Tonelli Lagrangian on $M$. If $\gamma:[a,b]\to M$ is a minimizer for $L$, then the Legendre transform of its speed curve
$t\mapsto \Leg(\gamma(t),\dot\gamma(t))$ is a C$^1$ solution of the Hamiltonian  vector field $X_H$ obtained from the Hamiltonian $H$ associated to $L$.

Moreover, if $L$ is a Tonelli Lagrangian, there exists a (partial) {\rm C}$^1$ flow $\phi^L_t$ on $TM$ such that every speed curve of an $L$-minimizer is a part of an orbit of $\phi^L_t$. This flow is called the Euler-Lagrange flow, is defined by
$$\phi^L_t=\Leg^{-1}\circ \phi^H_t \circ \Leg,$$
where $\phi^H_t$ is the partial flow of the {\rm C}$^1$ vector filed $X_H$.}
\end{theorem}
\begin{proof} If we write $(x(t),p(t))=\Leg(\gamma(t),\dot\gamma(t))$ then
$$
x(t)=\gamma(t)\quad \text{and} \quad
p(t)=\frac{\partial L}{\partial v}(\gamma(t),\dot\gamma(t)).
$$
By Theorem \ref{regularityminimizer}, $x(t)=\gamma(t)$ is C$^1$ with $\dot x(t)=\dot\gamma(t)$. The fact that $p(t)$ is C$^1$ follows again from Theorem \ref{regularityminimizer}, which also yields in local coordinates
$$
\dot p(t)=\frac{\partial L}{\partial x}(\gamma(t),\dot\gamma(t)).
$$
Since $(x(t),p(t))=\Leg(\gamma(t),\dot\gamma(t))$, we conclude from Proposition
\ref{propertiesH} that $t\mapsto (x(t),p(t))$ satisfies the ODE
\begin{equation*}
\left\{
\begin{array}{l}
\dot x=\dfrac{\partial H}{\partial p}(x,p)\\
\dot p=-\dfrac{\partial H}{\partial x}(x,p).
\end{array}
\right.
\end{equation*}
Therefore the Legendre transform of the speed curve of a minimizer is a solution of the Hamiltonian vector field $X_H$.

If $L$ is a Tonelli Lagrangian, by Proposition \ref{propertiesH} the Hamiltonian $H$ is C$^2$. Therefore the vector field $X_H$ is C$^1$, and it defines a (partial) C$^1$ flow $\phi^H_t$. The rest follows from what was obtained above and the fact that the Legendre transform is C$^1$.
\end{proof}
We recall the following definition
\begin{definition}[Energy]{\rm If $L$ is a {\rm C}$^1$ Lagrangian on the manifold $M$,
its \textit{energy} $E:TM \rightarrow \R$ is defined by
$$
E(x,v) =H\circ \Leg(x,v)=
\frac{\partial L}{\partial v}(x,v)(v)-L(x,v).
$$}
\end{definition}
\begin{corollary}
\label{corConsEnergy}[Conservation of Energy]{\sl If $L$ is a {\rm
C}$^1$ Lagrangian on the manifold $M$, and $\gamma:[a,b]\to M$ is
a {\rm C}$^1$ minimizer for $L$, then the energy $E$ is constant
on the speed curve
$$
s \mapsto (\gamma(s),\dot \gamma(s)).
$$}
\end{corollary}
\begin{proof} In fact $E(\gamma(s),\dot\gamma(s))=H\circ \Leg(\gamma(s),\dot\gamma(s))$. But $s\mapsto\Leg(\gamma(s),\dot\gamma(s))$ is a solution of
the vector field $H$, and the Hamiltonian $H$ is constant on orbits of $X_H$.
\end{proof}
\begin{proposition}
\label{propEnergyCpt}{\sl If $L$ is a weak Tonelli Lagrangian on
the manifold $M$, then for every compact subset $K \subset M$, and
every $C <+\infty$, the set
$$
\{(x,v) \in TM \mid x \in K, \ E(x,v) \leq C\}
$$
is compact, i.e.\ the map $E:TM\to\R$ is proper on every subset of the form $\pi^{-1}(K)$, where $K$ is a compact subset of $M$.}
\end{proposition}
\begin{proof} Since $E=H\circ \Leg$, this follows from the fact that $H$ is proper
and $\Leg$ is a homeomorphism.
\end{proof}
\begin{proposition}\label{estimationspeed}{\sl Let $L$ be a weak Tonelli Lagrangian on $M$.
Suppose $K$ is a compact subset of $M$, and $t>0$. Then we can find a compact subset $\tilde K \subset M$ and
a finite constant $A$, such that every minimizer $\gamma:[0,t] \rightarrow M$ with $\gamma(0), \gamma (t) \in K$
satisfies $\gamma([0,t]) \subset \tilde K$
and $\lVert \dot \gamma(s) \rVert_{\gamma(s)} \leq A$ for every $s \in [0,t]$.}
\end{proposition}
\begin{proof}
We will use as a distance $d$ the one coming from the complete Riemannian metric.
All finite closed balls in this distance are compact (Hopf-Rinow theorem).
We choose $x_0 \in K$, and $R$ such that $K \subset B(x_0,R)$ (we could take $R=\operatorname{diam}(K)$, the diameter of $K$).
We now pick $x,y \in K$. If $\alpha:[0,t] \rightarrow M$ is a geodesic with $\alpha(0)=x$, $\alpha(t)=y$
and whose length is $d(x,y)$ (such a geodesic exists by completeness), the inequality
$$
d(x,y) \leq d(x,x_0) + d(x_0,y) \leq 2R
$$
implies that $\alpha([0,t]) \subset \bar B(x_0,3R)$.
Moreover $\lVert \dot \alpha(s) \rVert_{\alpha(s)} = d(x,y)/t \leq 2R/t$ for every $s \in [0,t]$.
By compactness, the Lagrangian $L$ is bounded on the set
$$
\KK =\{(z,v) \in TM \mid z \in \bar B(x_0,3R), \ \lVert v \rVert_z \leq 2R/t \}.
$$
We call $\theta$ an upper bound of $L$ on $\KK$. Obviously the action of $\alpha$ on $[0,t]$ is less
than $t\theta$, and therefore if $\gamma:[0,t] \rightarrow M$ is a minimizer with $\gamma(0), \gamma (t) \in K$, we get
$\int_0^t L(\gamma(s),\dot \gamma(s)) \,ds\leq t\theta$. Using condition
(c) on the Lagrangian $L$ and what we obtained above, we see that
$$
Ct+\int_0^t \lVert \dot \gamma(s) \rVert_{\gamma(s)} \,ds \leq t\theta.
$$
It follows that we can find $s_0 \in [0,t]$ such that
$$
\lVert \dot \gamma(s_0) \rVert_{\gamma(s_0)} \leq \theta -C.
$$
Moreover
$$
\gamma([0,t]) \subset \bar B(\gamma(0),t(\theta -C)) \subset \bar B(x_0,R+t(\theta -C)).
$$
We set $\tilde K = \bar B(x_0,R+t(\theta -C))$. If we define
$$
\theta_1 =\sup \{ E(z,v) \mid (z,v) \in TM, \ z \in \tilde K,\ \lVert v \rVert_z \leq \theta -C\},
$$
we see that $\theta_1$ is finite by compactness. Moreover $E(\gamma(s_0),\dot \gamma(s_0)) \leq \theta_1$.
But, as mentioned earlier, the energy $E(\gamma(s),\dot \gamma(s))$ is constant on the curve.
This implies that the speed curve
$$
s \mapsto (\gamma(s),\dot \gamma(s))
$$
is contained in the compact set
$$
\tilde \KK =\{(z,v) \in TM \mid z \in \tilde K,\ E(z,v) \leq \theta_1 \}.
$$
Observing that the set $\tilde \KK$ does not depend on $\gamma$, this finishes the proof.
\end{proof}
\subsection{Lagrangian costs and semi-concavity}
\begin{definition}[Costs for a Lagrangian]\label{CostLagrangian}{\rm Suppose $L:TM\to\R$ is a Lagrangian on the connected manifold $M$, which is bounded from below. For $t>0$,  we define the \textit{cost} $c_{t,L}:M \times M \rightarrow \R$ by
$$
c_{t,L}(x,y) =\inf_{\gamma(0)=x,\gamma(t)=y} \mathbb{A}_L(\gamma)
$$
where the infimum is taken over all the absolutely continuous curves
$\gamma:[0,t]\to M$, with $\gamma(0)=x$, and $\gamma(t)=y$, and
$\mathbb{A}_L(\gamma)$ is the action
$\int_0^t L(\gamma(s),\dot \gamma(s)) \,ds$ of $\gamma$.}
\end{definition}
Using a change of variable in the integral defining the action, it is not difficult to see  that $c_{t,L}=c_{1, L^t}$ where the Lagrangian $L^t$ on $M$ is defined by
$L^t(x,v)=t L(x,t^{-1}v)$. Observe that $L^t$ is a (weak) Tonelli Lagrangian if $L$ is.
\begin{theorem}\label{costsemiconc}{\sl Suppose that $L:TM\to \R$
is a weak Tonelli Lagrangian.
Then, for every $t>0$, the
cost $c_{t,L}$ is locally semi-concave on $M \times M$.
Moreover, if the derivative of $L$ is locally Lipschitz, then $c_{t,L}$ is locally semi-concave with a linear modulus.

In particular, if $L$ is a Tonelli Lagrangian for every $t>0$, the
cost $c_{t,L}$ is locally semi-concave on $M \times M$ with a linear modulus.}
\end{theorem}
\begin{proof} By the remark preceding the statement of the theorem, it suffices to prove this for $c=c_{1,L}$.
Let $n$ be the dimension of $M$. Choose two charts $\varphi_i: U_i \stackrel{\sim}{\longrightarrow} \R^n$, $i=0,1$, on $M$.
We will show that
$$
(\tilde x_0,\tilde x_1) \mapsto c(\varphi_0^{-1}(\tilde x_0),\varphi_1^{-1}(\tilde x_1))
$$
is semi-concave on $\stackrel{\circ}{\B} \times \stackrel{\circ}{\B}$, where $\B$ is the closed Euclidean unit ball of center $0$ in $\R^n$.
By Proposition \ref{estimationspeed}, we can find a constant $A$ such that for every minimizer $\gamma:[0,1] \rightarrow M$,
with $\gamma(i) \in \varphi_i^{-1}(\B)$, we have
$$
\forall s \in [0,1], \quad \lVert \dot \gamma(s) \rVert_{\gamma(s)} \leq A.
$$
We now pick $\delta > 0$ such that for all $z_1,z_2 \in \R^n$, with $\Vert z_1\rVert_{\rm euc} \leq 1$, $\Vert z_2\rVert_{\rm euc}=2$,
$$
d(\varphi_i^{-1}(z_1),\varphi_i^{-1}(z_2)) \geq \delta, \quad i=0,1,
$$
where $\Vert \cdot \rVert_{\rm euc}$ denote the Euclidean norm.
Then we choose $\e > 0$ such that $A\e < \d$. It follows that
$$
\gamma([0,\e]) \subset \varphi_0^{-1}\bigl(2\stackrel{\circ}{\B}\bigr)\text{ and }
\gamma([1-\e,1]) \subset \varphi_1^{-1}\bigl(2\stackrel{\circ}{\B}\bigr).
$$
We set $\tilde x_i=\varphi_i(\gamma(i))$, $i=0,1$. For $h_0, h_1 \in \R^n$ we can define
$\tilde \gamma_{h_0} :[0,\e] \rightarrow \R^n$ and $\tilde \gamma_{h_1} :[1-\e,1] \rightarrow \R^n$ as
$$
\tilde \gamma_{h_0}(s)= \frac{\e - s}{\e}h_0 + \varphi_0(\gamma(s)), \ \ \ 0 \leq s \leq \e,
$$
$$
\tilde \gamma_{h_1}(s)= \frac{s-(1-\e)}{\e}h_1 + \varphi_1(\gamma(s)), \ \ \ 1-\e \leq s \leq 1.
$$
We observe that when $h_0=0$ (or $h_1=0$) the curve coincide with $\gamma$.
Moreover $\tilde \gamma_{h_0}(0)=\tilde x_0 + h_0$, $\tilde \gamma_{h_1}(1)=\tilde x_1 + h_1$.
We suppose that $\lVert h_i \rVert_{\rm euc} \leq 2$.
In that case the images of both $\tilde \gamma_{h_0}$ and  $\tilde \gamma_{h_0}$ are
contained in $4\stackrel{\circ}{\B}$ and
$$
\lVert \dot {{\tilde \gamma}}_{h_i}(s)\rVert_{\rm euc} \leq \lVert h_i \rVert_{\rm euc} + \lVert (\varphi_i \circ \gamma)'(s) \rVert_{\rm euc}
\leq 2 + \lVert (\varphi_i \circ \gamma)'(s) \rVert_{\rm euc}.
$$
Since we know that the speed of $\gamma$ is bounded in $M$, we can find a constant $A_1$ such that
$$
\forall s \in [0,\e], \quad \lVert \dot {{\tilde \gamma}}_{h_0}(s)\rVert_{\rm euc} \leq A_1,
$$
$$
\forall s \in [1-\e,1], \quad \lVert \dot {{\tilde \gamma}}_{h_1}(s)\rVert_{\rm euc} \leq A_1.
$$
To simplify a little bit the notation, we define the Lagrangian $L_i:\R^n \times \R^n \rightarrow \R$ by
$$
L_i(z,v)=L(\varphi_i^{-1}(z),D[\varphi_i^{-1}](v)).
$$
If we concatenate the three curves $\varphi_0^{-1} \circ \tilde \gamma_{h_0}$, $\gamma|_{[\e,1-\e]}$ and
$\varphi_1^{-1} \circ \tilde \gamma_{h_1}$, we obtain a curve in $M$ between
$\varphi_0^{-1}(\tilde x_0 + h_0)$ and $\varphi_1^{-1}(\tilde x_1 + h_1)$, and therefore
\begin{multline*}
c\left(\varphi_0^{-1}(\tilde x_0 + h_0),\varphi_1^{-1}(\tilde x_1 + h_1)\right)
\leq \int_0^\e L_0(\tilde \gamma_{h_0}(t),\dot {\tilde \gamma}_{h_0}(t)) \,dt \\
+ \int_\e^{1-\e} L(\gamma(t),\dot \gamma(t)) \,dt
+ \int_{1-\e}^1 L_1(\tilde \gamma_{h_1}(t),\dot {\tilde \gamma}_{h_1}(t)) \,dt.
\end{multline*}
Hence
\begin{multline*}
c\left(\varphi_0^{-1}(\tilde x_0 + h_0),\varphi_1^{-1}(\tilde x_1 + h_1)\right)
- c\left(\varphi_0^{-1}(\tilde x_0),\varphi_1^{-1}(\tilde x_1)\right) \\
\leq \int_0^\e\left[ L_0(\tilde \gamma_{h_0}(t),\dot {\tilde \gamma}_{h_0}(t))
- L_0(\varphi_0 \circ \gamma(t),(\varphi_0 \circ \gamma)'(t))\right] \,dt\\
+ \int_{1-\e}^1 \left[L_1(\tilde \gamma_{h_1}(t),\dot {\tilde \gamma}_{h_1}(t))
- L_1(\varphi_1 \circ \gamma(t),(\varphi_1 \circ \gamma)'(t))\right] \,dt.
\end{multline*}
We now call $\omega$ a common modulus of continuity for the derivative $DL_0$ and $DL_1$ on the compact set $\bar{B}(0,4)\times \bar B(0,A_1)$. Here $DL_0$ and $DL_1$ denote the total derivatives of $L_0$ and $L_1$, i.e.~with respect to all variables. When $L$ has  a derivative which is locally Lipschitz, then  $DL_0$ and $DL_1$ are also locally Lipschitz on $\R^n\times\R^n$, and the modulus $\omega$ can be taken linear. Since $\tilde \gamma_{h_i}(s) \in \stackrel{\circ}{B}(0,4)$ and $\lVert \dot {{\tilde \gamma}}_{h_i}(s) \rVert \leq A_1$,
we get the estimate
\begin{align*}
c\bigl(\varphi_0^{-1}(\tilde x_0 + h_0),&\varphi_1^{-1}(\tilde x_1 + h_1)\bigr)
- c\bigl(\varphi_0^{-1}(\tilde x_0),\varphi_1^{-1}(\tilde x_1)\bigr) \\
&\leq \int_0^\e DL_0\left(\varphi_0 \circ \gamma(t),(\varphi_0 \circ \gamma)'(t) \right)
\left(\frac{\e - t}{\e} h_0,-\frac{1}{\e}h_0 \right)\,dt\\
&+ \int_{1-\e}^1 D L_1\left(\varphi_1 \circ \gamma(t),(\varphi_1 \circ \gamma)'(t) \right)
\left(\frac{t-(1-\e)}{\e} h_1,\frac{1}{\e}h_1 \right)\,dt\\
&+ \omega\left(\frac{1}{\e} \lVert h_0\rVert_{\rm euc}\right) \frac{1}{\e} \lVert h_0\rVert_{\rm euc}
+ \omega\left(\frac{1}{\e} \lVert h_1\rVert_{\rm euc}\right) \frac{1}{\e} \lVert h_1\rVert_{\rm euc}.
\end{align*}
We observe that the sum of the first two terms in the right hand side is linear, while the sum of the last two is bounded by
$$
\frac{1}{\e} \omega\left(\frac{1}{\e} \lVert (h_0,h_1)\rVert_{\rm euc}\right) \lVert (h_0,h_1)\rVert_{\rm euc}.
$$
Therefore we obtain that
$$
(\tilde x_0,\tilde x_1) \mapsto c\left(\varphi_0^{-1}(\tilde x_0),\varphi_1^{-1}(\tilde x_1)\right)
$$
is semi-concave for the modulus $\tilde \omega(r)=\frac{1}{\e}\omega\left(\frac{1}{\e}r \right)$ on
$\stackrel{\circ}{\B} \times \stackrel{\circ}{\B}$, as wanted.
\end{proof}

\begin{corollary}\label{ComputationSuperdifferential}
{\sl If $L$ is a weak Tonelli Lagrangian on the connected manifold $M$, then, for every $t>0$,
a superdifferential of $c_{t,L}(x,y)$ at $(x_0,y_0)$ is given by
$$
(w_0,w_1) \mapsto \frac{\partial L}{\partial v}(\gamma(t),\dot \gamma(t))(w_1) -
\frac{\partial L}{\partial v}(\gamma(0),\dot \gamma(0))(w_0),
$$
where $\gamma:[0,t] \rightarrow M$ is a minimizer for $L$  with
$\gamma(0)=x_0$, $\gamma(t)=y_0$, and $(w_0,w_1)\in T_xM\times T_yM=T_{(x,y)}(M\times M)$.}
\end{corollary}
\begin{proof} Again we will do it only for $t=1$.
If we use the notation introduced in the previous proof, we see that a superdifferential of
$$
(\tilde x_0,\tilde x_1) \mapsto c\left(\varphi_0^{-1}(\tilde x_0),\varphi_1^{-1}(\tilde x_1)\right)
$$
is given by
$$
(h_0,h_1) \mapsto l_0(h_0) + l_1(h_1),
$$
where
\begin{multline}
\label{formulal0}
l_0(h_0)= -\int_0^\e \Bigl[\frac{t-\e}{\e}\frac{\partial L_0}{\partial x}\left(\varphi_0 \circ \gamma(t),(\varphi_0 \circ \gamma)'(t) \right)(h_0)\\
+ \frac{1}{\e} \frac{\partial L_0}{\partial v}\left(\varphi_0 \circ \gamma(t),(\varphi_0 \circ \gamma)'(t) \right)(h_0) \Bigr]\,dt,
\end{multline}
\begin{multline*}
l_1(h_1)= \int_{1-\e}^1 \Bigl[\frac{t-(1-\e)}{\e}\frac{\partial L_1}{\partial x}\left(\varphi_1 \circ \gamma(t),(\varphi_1 \circ \gamma)'(t) \right)(h_1)\\
+ \frac{1}{\e} \frac{\partial L_1}{\partial v}\left(\varphi_1 \circ \gamma(t),(\varphi_1 \circ \gamma)'(t) \right)(h_1)\Bigr]\,dt.
\end{multline*}
By Theorem \ref{regularityminimizer}, the curve $t \mapsto \varphi_0 \circ \gamma(t)$ is a
C$^1$ extremal of $L_0$ and it satisfies  the Euler-Lagrange equation
$$
\frac{d}{dt}\frac{\partial L_0}{\partial v}\left(\varphi_0 \circ \gamma(t),(\varphi_0 \circ \gamma)'(t) \right)
\frac{\partial L_0}{\partial x}\left(\varphi_0 \circ \gamma(t),(\varphi_0 \circ \gamma)'(t) \right).
$$
Using this identity in (\ref{formulal0}) an integrating by parts, we get
\begin{multline*}
l_0(h_0)=-\frac{\partial L_0}{\partial v}\left(\varphi_0 \circ \gamma(0),(\varphi_0 \circ \gamma)'(0) \right)(h_0).\\
\end{multline*}
This means that $l_0$, reinterpreted on $T_{x_0}M$ rather than on $\R^n$, is given by
$-\frac{\partial L}{\partial v}\left(\gamma(0),\dot \gamma(0) \right)$.
The treatment for $l_1$ is the same.
\end{proof}
We have avoided the first variation formula in the proof of Corollary
\ref{ComputationSuperdifferential}, because this is usually proven
for $C^2$ variation of curves and $C^2$ Lagrangians.
Of course, our argument to prove this Corollary is basically a proof for the first variation formula for $C^1$ Lagrangians. This is of course already known and the proof is the standard one.
\subsection{The twist condition for costs obtained from Lagrangians}

\begin{lemma}\label{lemmaUC}{\sl Let  $L$ be a weak Tonelli Lagrangian on the connected manifold $M$. Suppose that $L$ satisfies the following condition:
\begin{itemize}
\item[\rm (UC)] If $\gamma_i:[a_i,b_i]\to M, i=1,2$ are two $L$-minimizers such that
$\gamma_1(t_0)=\gamma_2(t_0)$ and $\dot\gamma_1(t_0)=\dot\gamma_2(t_0)$, for some $t_0\in [a_1,b_1]\cap [a_2,b_2]$, then $\gamma_1=\gamma_2$ on the whole interval $[a_1,b_1]\cap [a_2,b_2]$.
\end{itemize}
Then, for every $t>0$, the cost $c_{t,L}:M\times M\to\R$ satisfies the left (and the right) twist condition of Definition \ref{twistcondition}.

Moreover, if $(x,y)\in {\cal D}(\Lambda ^l_{c_{t,L}})$, then we have:
\begin{itemize}
\item[\rm (i)]
there is a unique $L$-minimizer $\gamma:[0,t]\to M$ such that $x=\gamma(0)$, and
 $y=\gamma(t)$;
 \item[\rm (ii)]  the speed $\dot\gamma(0)$ is uniquely determined by the equality
 $$
\frac{\partial c_{t,L}}{\partial x}(x,y)=-\frac{\partial L}{\partial v}(x,\dot\gamma(0)).
$$
\end{itemize}}
\end{lemma}
\begin{proof}
We first prove part (ii).
Pick $\gamma:[0,t]\to M$ an $L$-minimizer with $x=\gamma(0)$ and
$y=\gamma(t)$. From Corollary \ref{ComputationSuperdifferential} we obtain the equality
\begin{equation*}
\frac{\partial c_{t,L}}{\partial x}(x,y)=-\frac{\partial L}{\partial v}(x,\dot\gamma(0)). \tag{$*$}
\end{equation*}
Since the C$^1$ map $v\mapsto L(x,v)$ is strictly convex, the Legendre transform $v\in T_xM\mapsto \partial L/\partial v(x,v)$ is injective, and therefore $\dot\gamma(0)\in T_xM$ is indeed uniquely determined  by Equation ($*$) above. This proves (ii).

To prove statement (i), consider  another $L$-minimizer $\gamma_1:[0,t]\to M$ is  $x=\gamma_1(0)$. By what we just said, we also have
$$\frac{\partial c_{t,L}}{\partial x}(x,y)=-\frac{\partial L}{\partial v}(x,\dot\gamma_1(0)).$$
By the uniqueness already proved in statement (ii), we get
 $\dot\gamma_1(0)=\dot\gamma(0)$. It now follows from condition (UC) that $\gamma=\gamma_1$ on the whole interval $[0,t]$.

The twist condition follows easily. Consider $(x,y),(x,y_1)\in {\cal D}(\Lambda ^l_{c_{t,L}})$ such that
\begin{equation*}\frac{\partial c_{t,L}}{\partial x}(x,y)=\frac{\partial c_{t,L}}{\partial x}(x,y_1),\tag{$**$}.
\end{equation*}
By (i) there is a unique $L$-minimizer $\gamma:[0,t]\to M$ (resp.\ $\gamma_1:[0,t]\to M$) such that
$x=\gamma(0),y=\gamma(1)$ (resp.\ $x=\gamma_1(0),y_1=\gamma_1(1)$), and
$$
\frac{\partial c_{t,L}}{\partial x}(x,y)=-\frac{\partial L}{\partial v}(x,\dot\gamma(0))\quad
\text{and} \quad
\frac{\partial c_{t,L}}{\partial x}(x,y_1)=-\frac{\partial L}{\partial v}(x,\dot\gamma_1(0)).
$$
>From equation ($**$), and the injectivity of the Legendre transform of $L$, it follows that  $\dot\gamma_1(0)=\dot\gamma(0)$. From condition (UC) we get $\gamma=\gamma_1$ on the whole interval $[0,t]$. In particular, we obtain $y=\gamma(t)=\gamma_1(t)=y_1$.
\end{proof}
The next lemma is an easy consequence of Lemma \ref{lemmaUC} above.
\begin{lemma}\label{lemmaUC'}{\sl Let $L$ be a weak Tonelli Lagrangian on $M$. If we can find a continuous local flow $\phi_t$ defined on $TM$ such that:
\begin{itemize}
\item[\rm (UC')] for every $L$-minimizer
$\gamma:[a,b]\to M$, and every $t_1,t_2\in [a,b]$,
the point $\phi_{t_2-t_1}(\gamma(t_1),\dot\gamma(t_1))$ is defined and
$(\gamma(t_2),\dot\gamma(t_2))=\phi_{t_2-t_1}(\gamma(t_1),\dot\gamma(t_1))$,
\end{itemize}
then $L$ satisfies {\rm (UC)}. Therefore, for every $t>0$, the cost $c_{t,L}:M\times M\to\R$ satisfies the left twist (and the right) condition of Definition \ref{twistcondition}.

Moreover, if $(x,y)\in {\cal D}(\Lambda ^l_{c_{t,L}})$, then $y=\pi\phi_t(x,v)$, where $\pi:TM\to M$ is the canonical projection, and $v\in T_xM$ is uniquely determined by the equation
$$\frac{\partial c_{t,L_{r,g}}}{\partial x}(x,y)=-\frac{\partial L}{\partial v}(x,v).$$
The curve $s\in [0,t]\mapsto \pi\phi_s(x,v)$ is the unique $L$-minimizer $\gamma:[0,t] \to M$ with $\gamma(0)=x, \gamma(1)=y$.}
\end{lemma}
Note that the following proposition is contained in Theorem \ref{ConnectionWithH}.
\begin{proposition}\label{TwistForTonelli}{\sl If $L$ is a Tonelli Lagrangian,
then it satisfies condition {\rm (UC')} for the Euler Lagrange flow $\phi^L_t$.}
\end{proposition}
\begin{proposition}\label{RiemannianUFC'}{\sl Suppose $g$ is a complete Riemannian metric on the connected manifold $M$, and $r>1$. For a given $t>0$, the cost
$c_{t,L_{r,g}}$of  the weak Tonelli Lagrangian $L_{r,g}$, defined by
$$L_{r,g}(x,v)=\lVert v\rVert_x^r=g_x(v,v)^{r/2},$$
 is given by
$$c_{t,L_{r,g}}=t^{r-1}d_g^r(x,y),$$
where $d_g$ is the distance defined by the Riemannian metric.
The Lagrangian $L_{r,g}$ satisfies condition {\rm (UC')} of Lemma \ref{lemmaUC'} for the geodesic flow $\phi^g_t$ of $g$. Therefore its  cost $c_{t,L_{r,g}}$ satisfies the left (and the right) twist condition.
Moreover,  if $(x,y)\in {\cal D}(\Lambda ^l_{c_{t,L_{r,g}}})$, then $y=\pi\phi^g_t(x,v)$, where $\pi:TM\to M$ is the canonical projection, and $v\in T_xM$ is uniquely determined by the equation
$$\frac{\partial c_{t,L_{r,g}}}{\partial x}(x,y)=-\frac{\partial L_{r,g}}{\partial v}(x,v).$$
}
\end{proposition}
\begin{proof} Define $s$ by $1/s + 1/r =1$. Let $\gamma:[a,b]\to M$ be a piecewise C$^1$ curve. Denoting by $\ell_g(\gamma)$ the Riemannian length of $\gamma$, we can apply H\"older inequality to obtain
$$\int_a^b\lVert \gamma(s)\rVert_x\, ds\leq (b-a)^{1/s}\left(\int_a^b\lVert \gamma(s)\rVert_x^r\, ds\right)^{1/r},$$
with equality if and only if $\gamma$ is parameterized with $\lVert \gamma(s)\rVert_x$ constant, i.e. proportionally to arc-length. This of course implies
$$(b-a)^{-r/s}\ell_g(\gamma)^r\leq\int_a^b\lVert \gamma(s)\rVert_x^r\, ds,$$
with equality if and only if $\gamma$ is parameterized  proportionally to arc-length.
Since any curve can be reparametrized proportionally to arc-length and $r/s=r-1$, we conclude that
$$c_{t,L_{r,g}}(x,y)=t^{1-r}d_g(x,y)^r,$$
and that an $L_{r,g}$-minimizing curve has to minimize the length between its end-points. Therefore any $L_{r,g}$-minimizing curve is a geodesic and its speed curve is an orbit of the geodesic flow $\phi^g_t$. Therefore $L_{r,g}$ satisfies  condition {\rm (UC')} of Lemma \ref{lemmaUC'} for the geodesic flow $\phi^g_t$ of $g$. The rest of the proposition follows from Lemma \ref{lemmaUC'}.
\end{proof}


\begin{thebibliography}{99}

\bibitem{AbrahamMarsden}{\sc R. Abraham \& J.E. Marsden:} {\em Foundations of mechanics. Second edition, revised and enlarged.} (1978)
Benjamin/Cummings Publishing Co. Inc. Advanced Book Program,
Reading, Mass.

\bibitem{amb1}{\sc L. Ambrosio:} {\em Lecture notes on optimal transport problems,} in Mathematical Aspects of Evolving Interfaces, Lecture Notes in Math., {\bf 1812}, (2003) Springer-Verlag, Berlin/New York, 1-52.

\bibitem{amasbut}
{\sc L. Ambrosio, O. Ascenzi \& G. Buttazzo:}
{\em Lipschitz regularity for minimizers of integral functionals with highly discontinuous integrands.}
J. Math. Anal. Appl., {\bf 142} (1989), no. 2, 301-316.

\bibitem{amgisa} {\sc L. Ambrosio, N. Gigli \& G. Savar\'e:}
{\em Gradient flows in metric spaces and in the Wasserstein space of probability measures.}
Lectures in Mathematics, ETH Zurich, (2005) Birkh\"auser.

\bibitem{ambrosio-pratelli}{\sc L. Ambrosio \& A. Pratelli:} {\em Existence and stability results in the $L^1$ theory of optimal transportation, in Optimal Transportation and Applications.} Lecture Notes in Math., {\bf 1813}, (2003) Springer-Verlag, Berlin/New York, 123-160.

\bibitem{bangert} {\sc V. Bangert:}
{\em Analytische Eigenschaften konvexer Funktionen auf Riemannschen Manigfaltigkeiten.}
J. Reine Angew. Math., {\bf 307} (1979), 309-324.

\bibitem{bernbuf} {\sc P. Bernard \& B. Buffoni:}
{\em Optimal mass transportation and Mather theory.}
J. Eur. Math. Soc., {\bf  9} (2007),  no. 1, 85-121.

\bibitem{bernbuf2}
{\sc P. Bernard \& B. Buffoni:} {\em The Monge problem for supercritical Ma\~n\'e potential
on compact manifolds.} Adv. Math., {\bf 207} (2006), no. 2, 691-706.

\bibitem{brenier}{\sc Y. Brenier:} {\em Polar factorization and monotone rearrangement of vector-valued functions.} {Comm. Pure Appl. Math.}, {\bf44}
(1991), 375-417.

\bibitem{butgiahil}{\sc G. Buttazzo, M. Giaquinta \& S. Hildebrandt:} {\em One-dimensional variational problems.} Oxford Lecture Series in Mathematics and its Applications,
{\bf 15} (1998), Oxford Univ. Press, New York.

\bibitem{caffarelli-feldman-mccann}
{\sc L. Caffarelli, M. Feldman \& R.J. McCann:} {\em Constructing optimal maps for Monge's transport problem as a limit of strictly convex costs.} J. Amer. Math. Soc., {\bf 15} (2002), 1-26.

\bibitem{CaffarelliSalsa}{\sc L. A. Caffarelli \& S. Salsa, Editors:}
{\em Optimal transportation and applications.}
{Lecture Notes in Mathematics},
{\bf 1813}. {Lectures from the C.I.M.E.\ Summer School held in Martina
            Franca, September 2--8, 2001}, (2003)
Springer-Verlag, Berlin.

 \bibitem{cansin}
{\sc P.Cannarsa \& C.Sinestrari:}
{\em Semiconcave Functions, Hamilton-Jacobi Equations, and Optimal Control.}
Progress in Nonlinear Differential Equations and Their Applications, {\bf 58} (2004),
Birkh\"auser, Boston.

\bibitem{ClarkeSIAM}{\sc F.H. Clarke:}
{\em Methods of dynamic and nonsmooth optimization.} {CBMS-NSF Regional Conference Series in Applied Mathematics},
{\bf 57} (1989),
{Society for Industrial and Applied Mathematics (SIAM)},
{Philadelphia, PA}.

\bibitem{ClarkeETDS}{\sc F.H.
  Clarke:} {\em A Lipschitz regularity theorem.} {Ergodic Theory Dynam. Systems}, to appear (2007).

\bibitem{ClarkeVinter}{\sc F.H.
  Clarke \& R.B. Vinter:} {\em Regularity properties of solutions to the basic problem in the
  calculus of variations.} {Trans. Amer. Math. Soc.},
{\bf 289} (1985), 73-98.

\bibitem{cohn}
{\sc D.L. Cohn:} {\em Measure theory.}
{Birkh\"auser Boston}, (1980), Mass.

\bibitem{evans}{\sc L.C. Evans:} {\em Partial differential equations and Monge-Kantorovich mass transfer,} in R. Bott et al., editors, {\em Current Developments in Mathematics.} (1997), International Press, Cambridge,  26-78.

\bibitem{evans-gangbo}{\sc L.C. Evans \& W. Gangbo:} {\em Differential equations methods for the Monge�antorovich mass transfer problem.} Mem. Amer. Math. Soc. {\bf 137} (1999).

\bibitem{fathiToulouse}
{\sc A. Fathi:} {\em Regularity of $C\sp 1$ solutions of the Hamilton-Jacobi equation.}  Ann. Fac. Sci. Toulouse Math. (6), {\bf 12}  (2003), 479-516.

\bibitem{fathi}
{\sc A. Fathi:}
{\em Weak KAM theorems in Lagrangian Dynamics.}
Book to appear.

\bibitem{federer}
{\sc H. Federer:} {\em Geometric measure theory. Die Grundlehren der mathematischen Wissenschaften,} {\bf 153} (1969) Springer-Verlag New York Inc., New York.

\bibitem{feldman-mccann}
{\sc M. Feldman \& R.J. McCann:} {\em Monge's transport problem on a Riemannian manifold.}
Trans. Amer. Math. Soc., {\bf 354} (2002), 1667-1697.

\bibitem{figalli}
{\sc A. Figalli:}
{\em The Monge problem on non-compact manifolds.}
Rend. Sem. Mat. Univ. Padova, {\bf 117} (2007), 147-166.

\bibitem{forni-mather}
{\sc J.N. Mather \& G. Forni:} {\em Action minimizing orbits in
Hamiltonian systems, Transition to chaos in classical and quantum mechanics,
Montecatini Terme, 1991.}
Lecture Notes in Math., {\bf 1589} (1994), {Springer},{Berlin},
{92-186}.

\bibitem{gangbo}{\sc W. Gangbo:} {\em The Monge mass transfer problem and its applications,
  in Monge Amp\`ere equation: applications to geometry and
      optimization.} {Contemp. Math.}, {\bf 226}  (1999),
{Amer. Math. Soc.}, {Providence, RI}, {79-104}.

\bibitem{gangbo-mccann}{\sc W. Gangbo \& R.J. McCann:} {\em The geometry of optimal transportation.} Acta Math., {\bf 177}
(1996), 113-161.

\bibitem{HoferZehnder}{\sc H. Hofer \& E. Zehnder:} {\em Symplectic invariants and Hamiltonian dynamics.} Birkhuser Advanced Texts: Basler Lehrb\"ucher. (1994) Birkhuser Verlag, Basel.

\bibitem{kant1}
{\sc L.V. Kantorovich:}
{\em On the transfer of masses.}
Dokl. Akad. Nauk. SSSR, {\bf 37} (1942), 227-229.

\bibitem{kant2}
{\sc L.V. Kantorovich:}
{\em On a problem of Monge.}
Uspekhi Mat.Nauk., {\bf 3} (1948), 225-226.

\bibitem{knott-smith} {\sc M. Knott \& C.S. Smith:} {\em On the optimal mapping of distributions.} J. Optim. Theory Appl.,
{\bf 43} (1984), 39�9.

\bibitem{mather}
{\sc J.N. Mather:}
{\em Existence of quasiperiodic orbits for twist homeomorphisms of the annulus.} Topology,  {\bf 21}  (1982), 457-467.

\bibitem{mccann3}
{\sc R. McCann:}
{\em Polar factorization of maps on Riemannian manifolds.}
Geom. Funct. Anal., {\bf 11} (2001), no. 3, 589-608.

\bibitem{monge}
{\sc G. Monge:} {\em M\'emoire sur la Th\'eorie des D\'eblais et des Remblais.} Hist. de l'Acad. des Sciences de Paris (1781), 666-704.

\bibitem{Moserbook}{\sc J. Moser:}
{\em Selected chapters in the calculus of variations},
{Lectures in Mathematics ETH Z\"urich}, Lecture notes by Oliver Knill, (2003), Birkh\"auser Verlag,
{Basel}.

\bibitem{rachev-ruschendorf-article}
{\sc S.T. Rachev \& L. R\"uschendorf:}
{\em A characterization of random variables with minimum $L\sp 2$-distance.} J. Multivariate Anal.,
{\bf 32} (1990), {48-54}. Corrigendum in {J. Multivariate Anal.}, {\bf 34} (1990), {156}.

\bibitem{racrus}
{\sc S.T. Rachev \& L. R\"uschendorf:}
{\em Mass transportation problems.}
Vol. I: Theory, Vol. II: Applications.
Probability and its applications, Springer, (1998).

\bibitem{schteich}{\sc W. Schachermayer \& J. Teichmann:} {\em Characterization of optimal Transport Plans for the Monge-Kantorovich-Problem.}  Proceedings of the A.M.S., to appear.

\bibitem{sudakov}{\sc V.N. Sudakov:} {\em Geometric problems in the theory of infinite-dimensional probability distributions.} Proc. Steklov Inst. Math. {\bf 141} (1979), 1-178.

\bibitem{trudinger-wang}
{\sc N.S. Trudinger \& X.J. Wang:} {\em On the Monge mass transfer problem.} Calc. Var. Partial Differential Equations, {\bf 13} (2001), 19-31.

\bibitem{villani1}
{\sc C. Villani:}
{\em Topics in mass transportation.}
Graduate Studies in Mathematics, {\bf 58} (2004), American Mathematical Society, Providence, RI.

\bibitem{villani2}
{\sc C. Villani:}
{\em Optimal transport, old and new.}
Lecture notes, 2005 Saint-Flour summer school.

\bibitem{vinterbook}
{\sc R. Vinter:}
{\em Optimal control},
 {Systems \& Control: Foundations \& Applications}, (2000),
{Birkh\"auser Boston Inc.},
{Boston, MA}.

\end{thebibliography}
\end{document}